\documentclass[11pt]{article}
\usepackage[utf8]{inputenc}
\usepackage{amsfonts, amsmath}
\usepackage{amssymb}
\usepackage{amsthm,amsmath,amscd}
\usepackage{enumerate}
\usepackage{url}
\usepackage{amsbsy}
\usepackage[pdftex]{graphicx}
\usepackage[margin=25pt,font=small,labelfont=bf]{caption}
\usepackage[breaklinks,hyperfootnotes=false,bookmarks=false,colorlinks=true]{hyperref}
\usepackage{subfig}
\usepackage{mdframed}
\usepackage{color}
\usepackage[square,numbers,sort&compress]{natbib}
\newcommand{\defn}[1]{\textcolor{blue}{\emph{#1}}}

\graphicspath{{figs/}}

\def\qed{\hfill {\hfill $\Box$} \medskip}
\DeclareMathOperator{\good}{Good}
\DeclareMathOperator{\bad}{Bad}
\DeclareMathOperator{\sing}{Sing}
\DeclareMathOperator{\Fano}{Fano}

\newcommand{\RR}{\mathbb R}

\newcommand{\EE}{\mathbb E}
\newcommand{\R}{\mathbb R}

\newcommand{\bna}{\begin{eqnarray}}
\newcommand{\ena}{\end{eqnarray}}
\newcommand{\ba}{\begin{eqnarray*}}
\newcommand{\ea}{\end{eqnarray*}}
\newcommand{\bs}[1]{}
\newcommand{\ud}{\,\mathrm{d}}
\newcommand{\edgecard}{N}
\newcommand{\incidencegraph}{\Delta}
\newcommand{\graphautomorphism}{\rho}
\newcommand{\reflectiongroup}{W}

\DeclareMathOperator{\Aut}{Aut}

\newtheorem{theorem}{Theorem}[section]

\newtheorem{lemma}[theorem]{Lemma}
\newtheorem{proposition}[theorem]{Proposition}
\newtheorem{remark}[theorem]{Remark}
\newtheorem{definition}[theorem]{Definition}

\DeclareMathOperator{\real}{Real}

\DeclareMathOperator{\Dim}{Dim}
\DeclareMathOperator{\Gr}{Gr}
\DeclareMathOperator{\rank}{rank}

\newcommand{\CC}{{\mathbb C}}
\newcommand{\CS}{{\mathcal S}}
\newcommand{\QQ}{{\mathbb Q}}
\newcommand{\ZZ}{{\mathbb Z}}

\newcommand{\bl}{{\bf l}}

\newcommand{\LL}{{L_{2,4}}}

\newcommand{\LN}{{L_{d,n}}}
\def\p{{\bf p}}

\def\pn{\p =(\p_1, \dots, \p_{n}) }
\def\q{{\bf q}}
\def\f{{\bf f}}
\def\g{{\bf g}}
\def\h{{\bf h}}
\def\G{{\bf G}}
\def\D{{\bf D}}
\def\bm{{\bf m}}

\def\e{{\bf e}}

\def\E{{\bf E}}

\def\x{{\bf x}}
\def\y{{\bf y}}
\def\z{{\bf z}}

\def\I{{\bf I}}
\def\PP{{\mathbb{P} }}

\def\A{{\bf A}}
\def\B{{\bf B}}

\def\S{{\bf S}}
\def\T{{\bf T}}
\def\J{{\bf J}}
\def\X{{\bf X}}
\def\Y{{\bf Y}}
\def\Z{{\bf Z}}

\def\P{{\bf P}}

\newcommand{\genericpoint}{\bl}

\newcommand{\trans}[1]{{#1}^\top}

\title{
Linear Symmetries of the Unsquared
Measurement Variety}

\author{
Ioannis Gkioulekas, 
Steven J. Gortler,
Louis Theran,
and 
Todd Zickler}
\date{}

\usepackage[margin=1in]{geometry}
\begin{document}
\maketitle 

\begin{abstract}
We introduce a new family of algebraic 
varieties, $L_{d,n}$, which we call the
unsquared measurement varieties. 
This family is parameterized
by a number of points $n$ and a dimension $d$. 
These varieties arise naturally from problems in 
rigidity theory and distance geometry. In those applications, it can 
be useful to understand the group of linear automorphisms of $L_{d,n}$.
Notably, a result of Regge implies that $L_{2,4}$ has an unexpected linear automorphism.
In this paper,
we give a complete characterization of the linear
automorphisms of $L_{d.n}$ for all $n$ and $d$.  
We show, that apart from  $L_{2,4}$ the unsquared measurement
varieties have no unexpected automorphisms. Moreover, for $L_{2,4}$ we characterize
the full automorphism group.
\end{abstract}

\section{Introduction} 
\label{sec:intro}

Many questions in graph rigidity and distance geometry can be answered
by studying an object called the squared measurement variety.
Given a configuration $\p$ of $n \ge d+2$ ordered points in $\RR^d$, we can 
measure the $N:= \binom{n}{2}$ ordered squared Euclidean distances between each pair
of the points and consider this as a single point in $\RR^N$.
(We associate each point in $\p$ with a vertex of the complete graph, $K_n$, and each of the $N$ vertex pairs
with an edge of $K_n$. Under this association, we refer to each of the $N$ measurements as
the squared distance of an edge.)
If we take the union of the measurement points over all possible
configurations of $n$ points in $\RR^d$, we obtain a subset
of $\RR^N$ which we call the 
\defn{Euclidean squared measurement set} (of the complete graph $K_n$)
denoted as 
$M^{\EE}_{d,n}$.  
Under complexification and the Zariski-closure
of this measurement set, we obtain a variety which we call the 
\defn{squared measurement variety} (of the complete graph) denoted as 
$M_{d,n} \subset \CC^{\edgecard}$. (Formal definitions are below.) This has also been called the
Cayley-Menger variety.
This variety 
is linearly isomorphic to 
$\CS^{n-1}_d$,
the variety of complex, symmetric 
$(n-1)\times(n-1)$
matrices of rank $d$ or less, which is a well understood
variety.

In~\cite{BK1}, 
Boutin and Kemper showed that one can 
uniquely reconstruct (up to congruence)
a generic configuration $\p$ in $\RR^d$ 
given its   $N$ \emph{unlabeled} pairwise squared distances.
By unlabeled, we mean that we are not told which 
distance measurement
corresponds to 
which pair of points. This central
result has many applications in rigidity and distance geometry~\cite{dux1,echo}.
The key to their result is showing that there are no 
permutations of the coordinate axes of $\CC^N$ 
(called edge permutations)
that map $M_{d,n}$ to itself, except for the permutations that
are consistent with a permutation of the indices of the $n$ points.
Such  permutations are said to be \defn{induced by a vertex relabeling}.

One can, of course, 
expand the question and ask about 
all the non-singular linear maps 
on $\CC^N$ that map $M_{d,n}$ to 
itself.  We call these \defn{linear 
automorphisms} of $M_{d,n}$.
Due to the linear relationship between 
$M_{d,n}$ and $\CS^{n-1}_d$, classifying the 
linear automorphisms of $M_{d,n}$ boils down
to looking at the linear automorphisms
of $\CS^{n-1}_d$. 
This is a classical question, and it is well known
that the 
linear automorphisms of $\CS^{n-1}_d$ are 
all linear maps with a ``factored'' form $\trans{\B}\G\B$, where $\G\in \CS^{n-1}_d$ 
and $\B$ is any $(n-1)\times (n-1)$ non-singular matrix (see, e.g., \cite{cip2}).

An even more general and daunting distance geometry problem
is to reconstruct an $n$ point configuration in $\RR^d$ given 
an unlabeled set of  Euclidean 
lengths of $N$ \emph{paths} through the configuration~\cite{loops0,velt}. 
By path, we mean
an ordered sequence of vertices, and we define its length
to be the sum of the Euclidean 
lengths of each edge
along the path.  Importantly, a path-length is defined
as
a sum of Euclidean edge lengths, not a sum of squared Euclidean edge lengths.
As such, the relevant algebraic variety 
to study should
represent
lengths instead of squared lengths.  
Also, in the unlabeled setting,
we are not even  given the information as to which combinatorial
paths were measured in the first place. Thus we are not just
concerned with coordinate permutations of $\CC^n$ but
with more general linear maps acting on this variety
(such as those arising from sums of lengths).

To this end, we define 
the \defn{squaring map} $s(\cdot)$ be the map from $\CC^{\edgecard}$ 
onto $\CC^{\edgecard}$ that
acts by squaring each of the $\edgecard$ coordinates of a point.
We then define
the \defn{unsquared measurement variety} of $n$ points
in $d$ dimensions,
$L_{d,n}$, as the preimage of $M_{d,n}$ under the squaring map.
(Each point in $M_{d,n}$ has $2^{\edgecard}$ preimages in $L_{d,n}$, arising
through coordinate negations).
We prove below 
(Theorem~\ref{thm:Lvariety}) that for $d \ge 2$, the variety $L_{d,n}$ is irreducible (for $d=1$ it is
actually a reducible arrangement of linear subspaces).

In this paper, 
we wish to understand the set of linear automorphisms
of $L_{d,n}$, ie. the non-singular 
linear maps on $\CC^N$ that map $L_{d,n}$ to itself.
Any coordinate permutation that is 
induced by a vertex relabeling
must be an automorphism.
Also due to the squaring construction, any
coordinate negation will be an automorphism. 
We call the group of automorphisms generated 
by vertex relabelings and coordinate the 
\defn{signed vertex relabelings}.

By homogeneity,  any uniform scale on $\CC^N$
will also  be an automorphism. 
Let us call the group of automorphisms  generated by signed 
vertex relabelings and uniforms scalings
the \defn{expected} automorphisms of $L_{d,n}$.

We then ask: are there any ``unexpected'' linear automorphisms of $L_{d,n}$?
Recall that $M_{d,n}$ has many linear automorphisms that are not 
permutations of any type.  By analogy, there is no a priori restriction on
what the linear automorphisms of $L_{d,n}$ can be.

In, fact, $L_{2,4}$ \emph{does} have an 
unexpected linear automorphism.
Regge~\cite{regge} (see also, Roberts \cite{regge2}) 
showed that the following linear map always takes the
Euclidean lengths $l$ of the edges of any 4-point 
configuration in $\RR^2$ to those, $l'$, of some different 4-point
configuration in $\RR^2$.
\begin{equation}\label{eq:reggemap}\tag{$\star$}
\begin{aligned}
l'_{13} &=& l_{13} \\
l'_{24} &=& l_{24} \\
l'_{12} &=& (-l_{12}+ l_{23}+ l_{34}+ l_{14})/2 \\
l'_{23} &=& (l_{12}-l_{23}+ l_{34}+ l_{14})/2 \\
l'_{34} &=& (l_{12}+ l_{23}- l_{34}+ l_{14})/2 \\
l'_{14} &=& (l_{12}+ l_{23}+ l_{34}- l_{14})/2
\end{aligned}
\end{equation}
This ``Regge symmtry'' gives rise to an 
unexpected linear automorphism of $L_{2,4}$.  
So the plot has thickened.

The first main result of this paper is that $L_{2,4}$ is the only unsquared measurement variety with an unexpected
linear automorphism. 
\begin{theorem}\label{thm:main}
Let $d \ge 1$ and let $n \ge d+2$.  
Assume that $\{d,n\}\neq\{2,4\}$.
Then any linear automorphism 
$\A$ of $L_{d,n}$ of is a scalar multiple
of a signed vertex relabeling.
\end{theorem}
This theorem is proven by combining the three cases proven below in Theorems
\ref{thm:no-regges}, \ref{thm:no-regges3} and \ref{thm:l13-aut}.

The second main result of this paper is 
to fully characterize the group of linear automorphisms of
$L_{2,4}$.
The details for this statement require a few definitions.

\begin{definition}\label{def:L24-groups}
Define $\Aut(L_{2,4})$ to be the linear automorphisms of $L_{2,4}$.  Let the
group $\PP \Aut(L_{2,4})$ be induced on the equivalence classes
of $\A\in \Aut(L_{2,4})$ under the relation ``$\A'$ is a complex scale of $\A$''.

We also consider the real subgroup $\Aut_{\RR}(L_{2,4})$.  This has a 
counterpart $\PP \Aut_{\RR}(L_{2,4})$ of equivalence classes up to 
real scale, and  $\PP_{+} \Aut_{\RR}(L_{2,4})$, on equivalence classes defined 
up to \textit{positive} scale.  It is well-defined to refer to an element 
of $\PP_{+} \Aut_{\RR}(L_{2,4})$ as being non-negative, since any equivalence
class containing a non-negative $\A$ consists entirely of non-negative matrices.
\end{definition}

\begin{theorem}
\label{thm:main2}
The group $\PP \Aut(L_{2,4})$
is of order $11520=768\cdot 15$.
It is 
generated by linear automorphisms that are represented by
matrices with rational elements.

The group $\PP_{+} \Aut_{\RR}(L_{2,4})$ is of order $23040$
and is isomorphic to the Weyl group $D_6$.  The subset of 
non-negative elements of $\PP_{+} \Aut_{\RR}(L_{2,4})$ is a 
subgroup of order $24$ and acts by relabeling the vertices of $K_4$.
\end{theorem}
(This is proven as Theorem~\ref{thm:auto2} below.)

We will also see that 
The group $\PP_{+} \Aut_{\RR}(L_{2,4})$ is 
in fact generated by 
the edge permutations induced by vertex relabeling, 
sign flip matrices, and the one Regge symmetry of \eqref{eq:reggemap}.

Our proof of this second theorem is computer aided.

\begin{remark}\label{rem:dpt-mo-question}
That $\PP_{+} \Aut_{\RR}(L_{2,4})$ contains a subgroup 
isomorphic to $D_6$ is based on conversations with
Dylan Thruston (see \cite{T17}) and has antecedents 
in~\cite{doyle}.
See \cite{I19,W17} for other 
geometric connections.
\end{remark}

The central step for the proof of Theorem~\ref{thm:main} is understanding which linear projection maps acting on $L_{d,n}$
can have deficient dimensions. This is done in Theorem~\ref{thm:linImage} below. That result can also be of independent
interest in unlabeled rigidity problems~\cite{loops0}. Additionally, in Appendix~\ref{sec:fano}, we study the large linear subspaces
contained in $L_{2,4}$, which can also be of independent use in unlabeled rigidity~\cite{loops0}.

\subsection*{Acknowledgements}
We would like to thank Dylan Thurston
for numerous helpful
conversations and suggestions throughout this project. His input on Regge symmetries, and on the use of 
covering space maps was essential.
We also thank Brian Osserman for fielding numerous algebraic geometry
queries.

Steven Gortler was partially supported by NSF grant DMS-1564473. Ioannis Gkioulekas and Todd Zickler received support from the DARPA REVEAL program under contract no.~HR0011-16-C-0028.




\section{Measurement Varieties}\label{sec:varieties}

We start by establishing our basic terminology.
We relegate our needed definitions and theorems from
algebraic geometry to Appendix~\ref{sec:geometry}.

\begin{definition}\label{def:constants}
Fix positive integers  $d$  and $n$.
Throughout the paper, we will set $N:= \binom{n}{2}$, 
$C:=\binom{d+1}{2}$, and $D:= \binom{d+2}{2}$.

These constants appear often because they are, respectively, the 
number of pairwise distances between $n$ points, the dimension of the 
group of congruences in $\RR^d$, 
and the number of edges in a complete $K_{d+2}$ graph.
\end{definition}

\begin{definition}
\label{def:graph}
A \defn{configuration}, $\pn$ is a  sequence of $n$ points
in $\RR^d$. (If we want to talk about points in $\CC^d$, we will
explicitly call this a \defn{complex configuration}.) The affine span
of a configuration need not be all of $\RR^d$.

We think of the integers in $[1,\dots,n]$ as the 
\defn{vertices}
of an abstract complete graph $K_n$. An \defn{edge}, $\{i,j\}$, is an
unordered distinct pair of vertices. The complete edge set of $K_n$ has cardinality $N$.

Fixing a configuration $\p$ in $\R^d$, we define the 
\defn{length} of an edge
$\{i,j\}$ to be the Euclidean distance between the points
$\p_i$ and $\p_j$, a real number. 

\end{definition}

Next we will study the basic properties
of two related families of varieties, the squared and unsquared
measurement varieties. 

The squared variety is very well studied
in the literature, but the unsquared variety is much less so.
Since we are interested in integer sums of unsquared 
edge lengths, we wish  to understand the structure of this
unsquared variety.

\begin{definition}
\label{def:sms}
Let us index the coordinates of $\CC^N$ as $ij$, with
$i < j$ and both between $1$ and $n$.  We also fix an ordering 
on the $ij$ pairs to index the coordinates of $\CC^N$ as $i$ with 
$i$ between $1$ and $N$.\footnote{This ordering choice does not matter as long 
as we are consistent.  It is there to lets us switch between coordinates 
indexed by edges of $K_n$ and indexed using flat vector notation.  For $n=4$, $N=6$ we will 
use the order: $12,13,23,14,24,34$.}
\end{definition}

Let us begin with a 
\defn{complex configuration} $\p$ of $n$ points in $\CC^d$
with $d \geq 1$. We will always assume  $n \geq d+2$.
There are $\edgecard$ vertex pairs (edges), along which we can measure
the complex \emph{squared} length as 
\ba
m_{ij}(\p) := \sum_{k=1}^{d}(\p^k_i-\p^k_j)^2
\ea
where $k$ indexes over the $d$ dimension-coordinates. Here, we measure
complex squared length using the complex square operation with no
conjugation. We consider the vector $[m_{ij}(\p)]$ over all of the vertex 
pairs, with $i<j$,
as a single point in $\CC^{\edgecard}$, which we denote as $m(\p)$.

\begin{definition}
Let $M_{d,n}\subset \CC^{\edgecard}$ be the
the image of $m(\cdot)$ over all $n$-point complex configurations in $\CC^d$. 
We call this the \defn{squared measurement variety} of $n$ points
in $d$ dimensions.
\end{definition}
When $n \le (d+1)$, then $M_{d,n}= \CC^{\edgecard}$.

\begin{definition}
If we restrict the domain to be real configurations, then 
we call the image under $m(\cdot)$ the \defn{Euclidean squared measurement set} denoted as 
$M^{\EE}_{d,n} \subset \RR^{\edgecard}$.
This set has real dimension 
$dn-C$. 
\end{definition}

The following theorem reviews some basic facts.
Most of the  ideas are discussed 
in~\cite{ciprian}, but we include a detailed proof here for completeness and ease of reference.

\begin{theorem}
\label{thm:Mvariety}
Let $n \ge d+2$.
The set $M_{d,n}$
is linearly isomorphic to 
$\CS^{n-1}_d$,
the variety of complex, symmetric 
$(n-1)\times(n-1)$
matrices of rank $d$ or less.
Thus, $M_{d,n}$ is a variety.
It is irreducible.
Its dimension is $dn-C$.
Its singular set $\sing(M_{d,n})$ consists of squared measurements of configurations
with affine spans of dimension strictly less than $d$.
\end{theorem}

\begin{proof}
Such an isomorphism is developed in~\cite{YH38}
and further, for example, in~\cite{gower}, see also
\cite[Section 7]{cgr}. The basic idea is as follows. We can, w.l.o.g.,
translate the entire complex configuration $\p$ 
in $\CC^d$ such that the last
point $\p_n$ is at the origin. We can then think of this as a
configuration of $n-1$ vectors in $\CC^d$.  Any such complex
configuration gives rise to
a symmetric 
$(n-1)\times(n-1)$
complex Gram matrix (where no conjugation is used),
$G(\p)$, of rank at most
$d$.  Conversely, any symmetric complex matrix $\G$ of rank $d$ or
less can be (Tagaki) factorized, giving rise to a complex configuration of
$n-1$ vectors in $\CC^d$, which, along with the origin, gives us an $n$-point
complex configuration $\p$ so that $\G=G(\p)$.

With this in place, 
let $\varphi$ be the 
invertible linear map 
from the space of 
$(n-1)\times(n-1)$ 
symmetric complex
matrices $\G$, to $\CC^{\edgecard}$
(indexed by vertex pairs $ij$, with $i<j$)
defined as
$\varphi(\G)_{ij}:= G_{ii} + G_{jj} -2G_{ij}$
(where $G_{in}$ and $G_{nj}$ is interpreted as
$0$).
(For invertibility see~\cite[Lemma 7]{cgr}.) 

When $\G=G(\p)$ is the gram matrix of a 
complex configuration $\p$ in $\CC^d$, then $\varphi(\G)$ 
computes the squared edge lengths of $\p$.
Since every symmetric matrix of rank at most $d$ 
arises 
as the Gram matrix, $G(\p)$ from some complex
configuration $\p$ in $\CC^d$, we see that the image of $\varphi$ 
acting on 
$\CS^{n-1}_d$,
is contained in 
$M_{d,n}$.
Conversely, 
since every point in $M_{d,n}$ arises from 
a complex configuration $\p$, and $\p$ gives rise to a Gram matrix 
$G(\p)$, we see that the image of $\varphi$ acting on rank 
constrained matrices is onto $M_{d,n}$.
This gives us our isomorphism of varieties (Lemma~\ref{lem:bij}.)

Irreducibility of $M_{d,n}$ follows from the fact that 
it is the image of an affine space (complex configuration space)
under a polynomial (the squared-length map).
The dimension follows from the dimension of 
$\CS^{n-1}_d$ which is 
$d(n-1) - \binom{d}{2}$ (this is consistent with a degree of freedom count;
see e.g., ~\cite{sym} for details).

For the description of the singular set of determinantal 
varieties of rank-constrained matrices,
see for example~\cite[Page 184]{harris} 
(which can also be applied to
the symmetric case). Meanwhile, we know that
$\G=G(\p)$ has rank $<d$
iff $\p$ has a deficient affine span 
in $\CC^d$ (see for example~\cite[Lemma 26]{cgr}). 
For an explicit statement about the singular set of
$M_{d,n}$, see~\cite[Proposition 4.5]{ciprian}.

\end{proof}

\begin{remark}
We note, but will not need, the following:
For $d\ge 1$,
the smallest complex variety containing   
$M^{\EE}_{d,n}$ is $M_{d,n}$.
\end{remark}
We note the following minimal instances where $n=d+2$.
In these cases, the variety has codimension $1$.

The variety $M_{1,3} \subset \CC^3$ is defined by the vanishing of the \defn{simplicial volume determinant}, that is, the determinant of the following matrix
\ba
\begin{pmatrix}
2m_{13}& (m_{13} + m_{23} - m_{12})\\
(m_{13} + m_{23} - m_{12})& 2m_{23}\\
\end{pmatrix}
\ea
where we use $(m_{12}, m_{13}, m_{23})$ to represent the coordinates of $\CC^3$. This is the Gram matrix, $\varphi^{-1}(m(\p))$, described in the proof of Theorem~\ref{thm:Mvariety}. 

The variety $M_{2,4} \subset \CC^6$ is defined by the vanishing of the 
determinant of the matrix
\ba
\begin{pmatrix}
2m_{14}& (m_{14} + m_{24} - m_{12})& (m_{14} + 
m_{34} - m_{13})\\
(m_{14} + m_{24} - m_{12})& 2m_{24}& (m_{24} + m_{34} - m_{23})\\
(m_{14} + m_{34} - m_{13})& (m_{24} + m_{34} - m_{23})& 
2m_{34}
\end{pmatrix}.
\ea

The variety $M_{3,5} \subset \CC^{10}$ is defined by the vanishing of the determinant of the matrix
\ba
\begin{pmatrix}
2m_{15}& (m_{15} + m_{25} - m_{12})& (m_{15} + 
m_{35} - m_{13})&  (m_{15} + m_{45} - m_{14})\\
(m_{15} + m_{25} - m_{12})& 2m_{25}& (m_{25} + m_{35} - m_{23})&  
(m_{25} + m_{45} - m_{24})\\
(m_{15} + m_{35} - m_{13})& (m_{25} + m_{35} - m_{23})& 
2m_{35}&  (m_{35} + m_{45} - m_{34})\\
(m_{15} + m_{45} - m_{14})& (m_{25} + m_{45} - m_{24})&
 (m_{35} + m_{45} - m_{34})& 2m_{45}
\end{pmatrix}.
\ea
These same polynomial calculations can be done by 
constructing the Cayley-Menger determinants.

When $n>d+2$, then $M_{d,n}$ has higher codimension, and 
requires the simultaneous vanishing of more
than one minor, characterizing the rank $d$.

Next we move on to unsquared lengths.


\begin{definition}
Let the \defn{squaring map} $s(\cdot)$ be the map from $\CC^{\edgecard}$ 
onto $\CC^{\edgecard}$ that
acts by squaring each of the $\edgecard$ coordinates of a point.
Let $L_{d,n}$ be the preimage of $M_{d,n}$ under the squaring map.
(Each point in $M_{d,n}$ has $2^{\edgecard}$ preimages in $L_{d,n}$, arising
through coordinate negations).
We call this the \defn{unsquared measurement variety} of $n$ points
in $d$ dimensions.
\end{definition}

\begin{definition}
We can define the \defn{Euclidean length map}
of a real configuration $\p$ as
\ba
l_{ij}(\p) := \sqrt{\sum_{k=1}^{d}(\p^k_i-\p^k_j)^2}
\ea
where we use the positive square root.
We call the image of $\p$ under
$l$ the \defn{Euclidean unsquared measurement set} denoted as 
$L^{\EE}_{d,n} \subset \RR^{\edgecard}$.
Under the squaring map, we get
$M^{\EE}_{d,n}$. 
We denote by $l(\p)$,
the vector $[l_{ij}(\p)]$ over all vertex 
pairs. We may consider $l(\p)$ either as a point in 
the real valued $L^{\EE}_{d,n}$
or as a point in 
the complex variety $L_{d,n}$.
\end{definition}

Indeed, $L^{\EE}_{d,n}$
is the set we are often interested in for applications,
but it will be easier to work with the whole variety $L_{d,n}$. 

\begin{remark}
The locus of $\LL$ where the edge lengths of a triangle,
$(l_{12}, l_{13}, l_{23})$, are held fixed is studied in 
beautiful detail in~\cite{marco}, where this is shown to be a Kummer surface.
\end{remark}

The following theorem is the main result of this section.
\begin{theorem}
\label{thm:Lvariety}
Let $n \ge d+2$.
$\LN$ is a variety. It has pure dimension
$dn-C$.
Assuming that $d \geq 2$, we also have the following:
$\LN$ is irreducible.
\end{theorem}
The proof is in the next subsection.
The non-trivial  part will be showing irreducibility, which 
we will do in Proposition~\ref{prop:irr} below.
Indeed,
in one dimension, the variety $L_{1,3}$ is reducible.

\begin{remark}
We note, but will not need the following:
For $d\geq 2$,
the smallest complex variety containing   
$L^{\EE}_{d,n}$ is $L_{d,n}$.
\end{remark}

Returning to our minimal examples:
The variety $L_{1,3} \subset \CC^3$ is defined by the vanishing of the determinant of the
following matrix
\ba
\begin{pmatrix}
2l^2_{13}& (l^2_{13} + l^2_{23} - l^2_{12})
\\
(l^2_{13} + l^2_{23} - l^2_{12})& 2l^2_{23}
\end{pmatrix}
\ea
where we use $(l_{12}, l_{13}, l_{23})$ to represent the coordinates of $\CC^3$.

The variety $L_{2,4} \subset \CC^6$ is defined by the vanishing of the determinant of the matrix
\ba
\begin{pmatrix}
2l^2_{14}& (l^2_{14} + l^2_{24} - l^2_{12})& (l^2_{14} + 
l^2_{34} - l^2_{13})\\
(l^2_{14} + l^2_{24} - l^2_{12})& 2l^2_{24}& (l^2_{24} + l^2_{34} - l^2_{23})
\\
(l^2_{14} + l^2_{34} - l^2_{13})& (l^2_{24} + l^2_{34} - l^2_{23})& 
2l^2_{34}
\end{pmatrix}.
\ea

The variety $L_{3,5}\subset \CC^{10}$ is defined by the vanishing of the determinant of the matrix
\ba
\begin{pmatrix}
2l^2_{15}& (l^2_{15} + l^2_{25} - l^2_{12})& (l^2_{15} + 
l^2_{35} - l^2_{13})&  (l^2_{15} + l^2_{45} - l^2_{14})\\
(l^2_{15} + l^2_{25} - l^2_{12})& 2l^2_{25}& (l^2_{25} + l^2_{35} - l^2_{23})&  
(l^2_{25} + l^2_{45} - l^2_{24})\\
(l^2_{15} + l^2_{35} - l^2_{13})& (l^2_{25} + l^2_{35} - l^2_{23})& 
2l^2_{35}&  (l^2_{35} + l^2_{45} - l^2_{34})\\
(l^2_{15} + l^2_{45} - l^2_{14})& (l^2_{25} + l^2_{45} - l^2_{24})&
 (l^2_{35} + l^2_{45} - l^2_{34})& 2l^2_{45}
\end{pmatrix}.
\ea

\begin{figure}[ht]
	\begin{center}
		\includegraphics[width=2.0in]{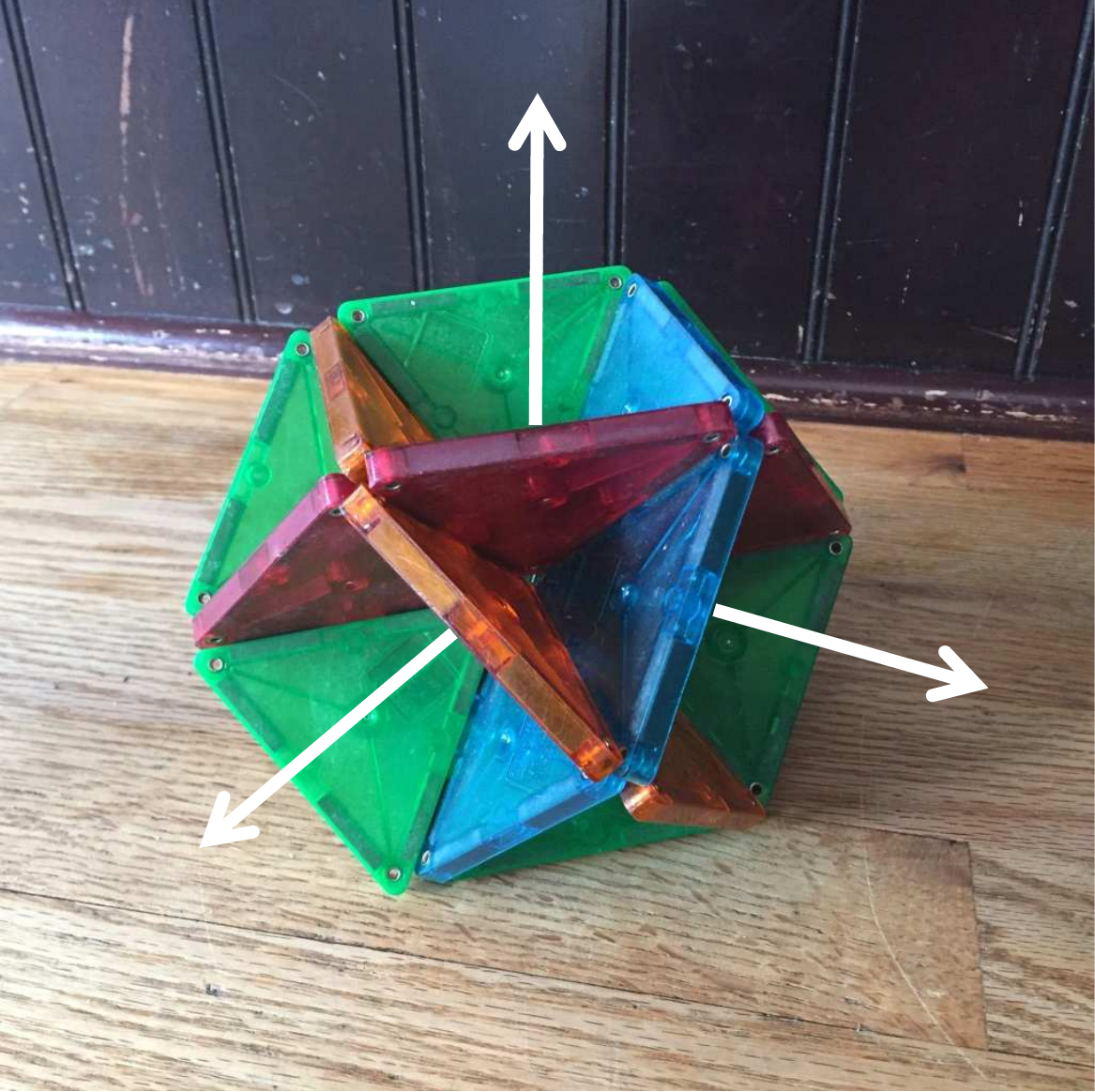}
	\end{center}
	\caption{A model of the real locus of $L_{1,3}$, a subset of $\RR^3$. It comprises $4$ planes. Coordinate axes are in white.}\label{fig:l13}
\end{figure}

\begin{remark}
It turns out that $L_{1,3}$ is reducible and consists of the four hyperspaces defined,
respectively, by the vanishing of one of the following equations:
\ba
l_{12} + l_{23} - l_{13} \\
l_{12} - l_{23} + l_{13} \\
-l_{12} +l_{23} + l_{13} \\
l_{12} + l_{23} + l_{13} 
\ea
This reducibility can make the one-dimensional case quite different from dimensions 2 and 3.

Notice that the first octant of the real locus of $3$ of these hyperspaces 
arises as the Euclidean lengths of a triangle in $\RR^1$ (that is,
these make up $L^{\EE}_{1,3}$). 
The specific hyperplane
is determined by the order of the $3$ points on the line.
\end{remark}

\subsection{Proof}
\label{sec:lproof}

We will now develop the proof of Theorem~\ref{thm:Lvariety}. The main issue will be
proving the irreducibility of $\LN$. 
The special case of $n=d+2$ follows from~\cite{cmirr}, but
we are interested in the general case, $n \ge d+2$.
The basic idea we will use is that a variety whose smooth locus
is connected must be irreducible.
More specifically, 
our strategy is to define a ``good'' locus of points in 
$\LN$, and show that this locus is connected,
made up of smooth points, and with its Zariski closure 
equal to $\LN$.
This, along with Theorem~\ref{thm:conIrr}, will prove irreducibility.
Note that when the word ``Zariski'' is not attached
to a topological term, you can interpret the term
in the standard topology.

We will show connectivity using a specific path construction. This
will rely centrally on the complex setting that we have placed ourselves in.  
Showing (algebraic) smoothness will mostly be a technical matter.

\begin{definition}
Let the \defn{zero} locus $Z$ of $\CC^N$ be the points where at least one
coordinate vanishes.

Let the \defn{bad} locus $\bad(M_{d,n})$ of $M_{d,n}$
be the union of its singular locus
$\sing(M_{d,n})$ together with the points in $M_{d,n}$ that are in $Z$.
We will call the remaining locus $\good(M_{d,n})$ \defn{good}.

Let the \defn{bad} locus $\bad(L_{d,n})$ of $L_{d,n}$ be the 
preimage of the bad locus of $M_{d,n}$ under the squaring map $s$.
We will call the remaining locus $\good(L_{d,n})$ \defn{good}. 

We refer to points on the good locus
as \defn{good} points, and analogously for \defn{bad} points.
\end{definition}

\begin{lemma}
\label{lem:Mcon}
$\good(M_{d,n})$ is path-connected.
\end{lemma}
\begin{proof}
Let $\bm_1$ and $\bm_2$ be any two 
good points in $M_{d,n}$. These correspond to 
two configurations $\p$ and $\q$. A path in configuration space, connecting $\p$ to $\q$,
will remain, under $m(\cdot)$, on $\good(M_{d,n})$
when the affine span of the configuration does not drop in dimension,
and no edge between any two points has zero squared length.
This can always be done, as we have $n \geq d+2$ points.
(This is even true for one-dimensional configurations
in the complex setting, as a zero squared length is a condition
that has complex-codimension of at least $1$, and thus the bad locus is non-separating.)
\end{proof}

We next record a lemma that follows from basic results
of covering space theory.  
See 
\cite[Sections 53, 54]{munk} for more details.
\begin{definition}
A \defn{path} 
$\tau$ on a space $X$ is a continuous  map from the unit interval to $X$.
A \defn{loop} is a path with $\tau(0)=\tau(1)$.
Let $p$ be a map from a space $\tilde{X}$ to $X$.
A \defn{lift} $\tilde{\tau}$ of $\tau$ (under $p$) is a map
such that $p(\tilde{\tau})=\tau$. It is a path on $\tilde{X}$.
\end{definition}
Intuitively, a lift is just tracing out 
the path $\tau$ in the preimage through $p$. In what follows, 
$\CC^\times$ is the punctured complex plane.
\begin{lemma}\label{lem:z2-cover}
Let $p$ be the map $\CC^\times\to \CC^\times$ given by $z\mapsto z^2$.
Let $x:=p(z)$.
A loop $\tau$ starting
at $x$ uniquely lifts to a loop $\tilde{\tau}$ starting at $z$
if $\tau$ winds 
around the origin an even number of times, and otherwise 
it lifts to a path that ends at $-z$.
\end{lemma}
\begin{proof}[Proof Sketch]
See 
\cite[Chapters 53, 54]{munk} for definitions.
The map $\CC^\times\to \CC^\times$ given by $z\mapsto z^2$
is a  covering map. 
Call the base $B$ and the cover $F$ and the 
covering map $p$.  
Each loop $\tau$ in $B$, starting at $x$, lifts uniquely to a path $\tilde{\tau}$
in $F$, starting at $z$.
The path $\tilde{\tau}$
ends at a uniquely defined point $z' \in p^{-1}(x)$
under the \defn{lifting correspondence}. 
In our case the fiber is $\{z,-z\}$. Moreover every $z'$
in the fiber can be reached under the lifting of some loop $\tau$  (see~\cite[Theorem 54.4]{munk}).

The fundamental group of the base is
$\pi_1(B) = 
\pi_1(\CC^\times)\cong \ZZ$.
The covering map determines an
induced map
$p_* : \pi_1(F)\to \pi_1(B)$. The image of the induced map
consists of 
loops that wind around the origin an even number of 
times in $F$ so it is isomorphic to $2\ZZ$.
The lifting correspondence induces a bijective  map from the 
group $\pi_1(B)/p_*(\pi_1(F))\cong \ZZ_2$ to the fiber above $x$, and (only) loops in 
$p_*(\pi_1(F))$ lift to loops in $F$.
(see~\cite[Theorem 54.6]{munk}). 

Thus, this lift, starting from $z$,
is a path from $z$ to $-z$ if and only if 
$\tau$ 
winds around the origin an odd number of 
times.
\end{proof}

\begin{figure}[t]
	\centering
	\def\svgwidth{1.5in}
	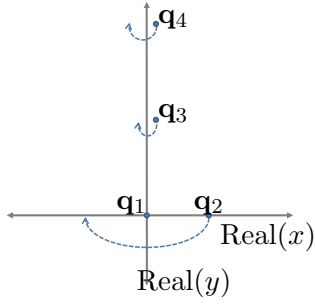
	\caption{Our gadget. The imaginary $x$-direction is coming out of the page. Our path ends with 
the reflection of the configuration $\q$ along the $x$-axis.}\label{fig:gadget}
\end{figure}

\begin{figure}[ht]
	\centering
	\def\svgwidth{0.8\textwidth}
	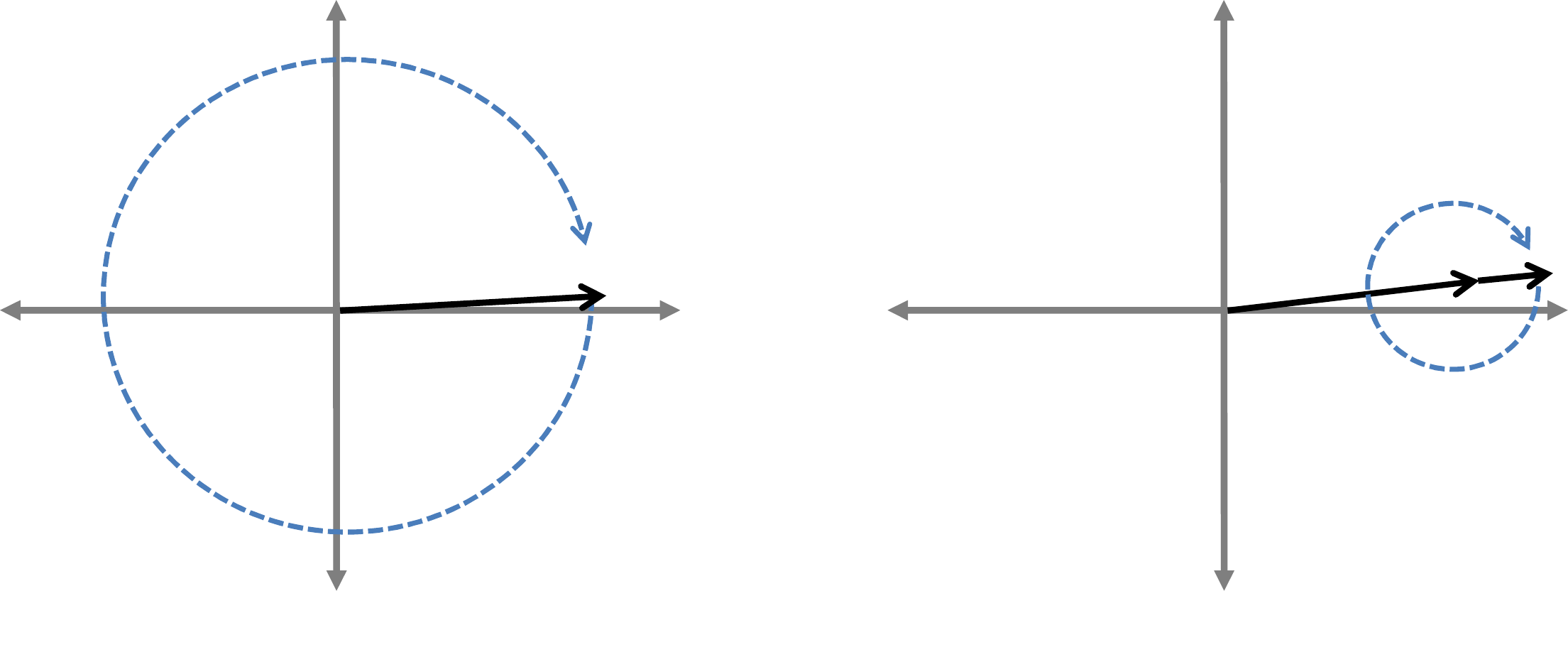
	\caption{Since the squared length along edge $\{1,2\}$ arises from its $x$ component, our path along this edge measurement winds once about the origin in $\CC$. For any other edge, the $x$ component of the squared distance is dominated by the other coordinates and the resulting path stays far from the origin in $\CC$.}\label{fig:wind}
\end{figure}

Looking at the product space $(\CC^{\times})^\edgecard$, we can also view the squaring map
$s$ as a covering map mapping this product space to itself, 
and we can apply Lemma~\ref{lem:z2-cover} coordinate-wise.

\begin{lemma}
\label{lem:Lneg}
Assume $d \ge 2$.
Suppose $\bl$ and $\bl'$ are two points in $\LN$ that 
differ only by a negation along one
coordinate. Then, there is a path that connects 
$\bl$ to $\bl'$ and stays in $\good(L_{d,n})$.
\end{lemma}
\begin{proof}
W.l.o.g., we will negate the coordinate corresponding to the edge
lengths between vertices $1$ and $2$.
But first, we need to develop a little gadget.

Let $\q$ be a special configuration with the following properties:
$\q_1$ is at the origin, $\q_2$ is placed one unit along the first axis of 
$\CC^d$; and the remaining points are arranged so that they all lie
within $\epsilon$ of the second axis in $\CC^d$, but such that
they are greater than one unit apart along the second axis
from each other and also from $\q_1$. 
(Note that this step requires that $d\geq 2$.)
Moreover we choose the remaining
points so that $\q$ has a full $d$-dimensional affine span.
This configuration has the following property:
the squared distances of all of the edges are dominated by the
contribution from the second coordinate, except for the squared distance
along the edge $\{1,2\}$, which is dominated by the contribution
from its first coordinate. See Figure~\ref{fig:gadget}.

Let $a(t)$ be the path in configuration space, parameterized by
$t \in [0,\pi]$ where, for each $i$, we multiply the first
coordinate of $\q_i$ by $e^{-t\sqrt{-1}}$. 
This path ends at 
$a(\pi)$, 
a configuration
which is a reflection of $\q$. 

Under $m$, this gives us
a loop $\tau:=m(a)$ in $M_{d,n}$ that starts and ends at the point
$\y:= m(\q)$. 
By construction, the loop
$\tau$ avoids any
singularities or vanishing coordinates.
Fixing one point $\z$ in $s^{-1}(\y)$, the loop $\tau$ lifts
to a path $\tilde{\tau}$ in $L_{d,n}$
that ends at some 
point $\z'$  in 
the fiber $s^{-1}(\y)$. Moreover, this path remains in $\good(L_{d,n})$. 

If we project $\tau$
onto the coordinate of $\CC^\edgecard$ corresponding
to the edge $\{1,2\}$, we see that the image
maps to a loop that
winds around the origin of $\CC$ exactly once. 
If we project this loop onto any
of the other coordinates, we obtain a loop  that cannot wind
about the origin of $\CC$ at all. See Figure~\ref{fig:wind}.
By Lemma~\ref{lem:z2-cover},
the lifted loop $\tilde{\tau}$ 
in $\LN$ must end at the point $\z'$ that arises from
$\z$ by negating the first coordinate.

Going now back to our problem, let $\p$ be any configuration such that 
$m(\p) = s(\bl)$. Let $w$ be a configuration path 
from $\p$ to our special $\q$. 
Let $\omega := m(w)$. From  Lemma~\ref{lem:Mcon} this path can 
be chosen to avoid any singular points or points where a coordinate
vanishes.
Let the concatenated path
$\sigma$ be $\omega^{-1} \circ \tau \circ \omega$.
This is a loop in $M_{d,n}$ that starts and ends at $m(\p)$.
The projection of $\sigma$
onto the coordinate of $\CC^\edgecard$ corresponding
to the edge $\{1,2\}$, defined by forgetting all other 
coordinates, 
winds around the origin
exactly once (any loops due to $\omega$ cancel out),
while the other coordinate projections are simply connected in 
$\CC^{\times}$ (any loops due to $\omega$ cancel out). Thus, fixing the point $\bl$ in $\LN$,
from Lemma~\ref{lem:z2-cover},
$\sigma$ must lift to a path $\tilde{\sigma}$ that ends at 
$\bl'$. Moreover, this path stays in the good locus.
\end{proof}

\begin{lemma}
\label{lem:Lcon}
For $d \ge 2$, $\good(L_{d,n})$ is path-connected.
\end{lemma}
\begin{proof}
Let $\bl_1$ and $\bl_2$ be two good points in $\good(L_{d,n})$.
Define $\bm_i := s(\bl_i)$. 
Let $\tau$ be a path 
in $M_{d,n}$
from 
$\bm_1$ to $\bm_2$ that avoids the singular set of $M_{d,n}$,
and such that no coordinate ever vanishes
(as guaranteed by~\ref{lem:Mcon}).
Fixing $\bl_1$, the path $\tau$
lifts to a path $\tilde{\tau}$ in $\LN$
that remains in the good locus 
and that connects $\bl_1$ to some point  $\bl_2'$ 
in the fiber 
$s^{-1}(s(\bl_2))$. 
The only remaining issue is that
$\bl_2'$ may have some of its coordinates negated from 
our desired target point $\bl_2$.
This can be solved by repeatedly applying the good negating paths
guaranteed by Lemma~\ref{lem:Lneg}.
\end{proof}

We now move on to the technical matters of smoothness.

\begin{lemma}
\label{lem:Asmooth}
Every point $\bl \in \good(\LN)$ is 
smooth
and with 
$\Dim_\bl(\LN)=dn-C$.
Every point in $\bad(\LN)-Z$ is singular.
\end{lemma}
\begin{proof}
Every good point  in $M_{d,n}$ is
(algebraically) smooth, and thus, from Theorem~\ref{thm:smpt}, is 
analytically smooth 
of dimension $dn-C$. 
Also, from Theorem~\ref{thm:smpt}, 
every singular point in $M_{d,n}$ is not analytically smooth.

The differential $\ud s$ of the squaring map $s$ on $\CC^\edgecard$ is
represented by an $\edgecard\times \edgecard$ Jacobian matrix $\J$ at
each point in $\CC^\edgecard$. At points in $\CC^\edgecard$ where none
of the coordinates vanish, $\J$ is invertible.
Thus, from the inverse function theorem,  every good point in $\LN$ is
analytically smooth 
of dimension $dn-C$. 
Also every bad point in $\LN-Z$ is not analytically smooth.

Again using Theorem~\ref{thm:smpt}, we have each good point 
(algebraically) smooth
and with 
$\Dim_\bl(\LN)=dn-C$.
Similarly, we also have that 
every bad point in $\LN-Z$ is singular.
\end{proof}
Note that there may be some bad points of $\LN$ in $Z$ that are still smooth.

\begin{remark}
The above lemma can be proven directly using more machinery from algebraic geometry.
In particular, away from $Z$, the squaring map from $\CC^N$ to itself is an 
``\'etale morphism''~\cite[page 18]{milne-etale}. 
This property transfers
to the map $s(\cdot)$ acting on $\LN-Z$,
as this property transfers under a ``base change''.
 The results then follows immediately.
\end{remark}

\begin{lemma}
\label{lem:closeG}
The Zariski closure of $\good(L_{d,n})$ is $\LN$.
\end{lemma}
\begin{proof}
Recall the following principle:
Given any point $z$ in $\CC^\times$, we can always find a 
neighborhood $B$
of $z^2$,
so that there is a 
well defined, single valued, continuous 
square root function
from $B$ to $\CC$, with $\sqrt{z^2}=z$.

Returning to our setting,
let $\bl$ be any point in $\LN$, and
let $\bm:=s(\bl)$ be its image in 
$M_{d,n}$ under the coordinate squaring map.
The good points of $M_{d,n}$ are  dense
in $M_{d,n}$. (Letting
$\bm=m(\p)$ for some $\p$,
there is always a nearby 
configuration $\p'$ with a full span and no edge with vanishing
squared length. Moreover, the 
map $m(\cdot)$ is continuous.)
Thus we can always find an arbitrarily close
point $\bm'$ that is in $\good(M_{d,n})$.

Next we argue that we can find a point $\bl'$
such that $s(\bl')=\bm'$ (putting it in $\good(\LN)$)
with $\bl'$ is arbitrarily close to $\bl$. Given $\bm'$,
in order to determine $\bl'$ we need to select a ``sign'' for the 
square-root on each coordinate $ij$.
When $l_{ij} \neq 0$ then using the above principle, 
we can pick a sign so that $l'_{ij}$ is near to 
$l_{ij}$.
When $l_{ij}=0$ then we can use any sign to 
obtain an 
$l'_{ij}$ that is sufficiently close to $0$.

Since this can be done for each
$\bl$, then $\LN$ is in 
the standard-topology closure of $\good(L_{d,n})$.

Thus, from Theorem~\ref{thm:Zdense}, 
$\LN$ is in 
the Zariski
closure of $\good(L_{d,n})$.
Since $\LN$ itself is closed
and contains $\good(L_{d,n})$,
we are done.



\end{proof}

\begin{lemma}
\label{lem:dim} 
Every component of 
$\LN$ is of  dimension equal to 
$dn-C$.
\end{lemma}
\begin{proof}
From Lemma~\ref{lem:Asmooth}
each good point has a local dimension of  $dn-C$.
Thus, the good locus is covered by a set of components of $\LN$, all of dimension 
$dn-C$.
The Zariski closure of $\good(L_{d,n})$ is $\LN$ (Lemma~\ref{lem:closeG}).
Thus, no new components need to be added during the Zariski closure.
\end{proof}

We can now prove irreducibility.

\begin{proposition}
\label{prop:irr}
For $d\geq2$,
$\LN$ is irreducible.
\end{proposition}
\begin{proof}
From  Lemma~\ref{lem:Asmooth}, all of the points in $\good(L_{d,n})$ are smooth. 
From Lemma~\ref{lem:Lcon}, $\good(L_{d,n})$ is path-connected,
and thus connected as a subspace of  $\CC^n$.

Now we show that all of $\good(L_{d,n})$
lies in a single irreducible component $V$ of 
$\LN$.  Fix an irreducible component $V$, such that $G = \good(L_{d,n})\cap V$
is non-empty.  Notice that $G$ is a closed subspace
of $\good(L_{d,n})$ (Theorem~\ref{thm:Zdense}). 
Now let $W$ be the union of 
all the remainining irreducible components 
of $L_{d,n}$.  By similar reasoning $H = W\cap \good(L_{d,n})$
is closed in $\good(L_{d,n})$.  

From Theorem~\ref{thm:conIrr}, $G$ and $H$ are disjoint.  
On the other hand, $V\cup W = L_{d,n}$, so $G\cup H = \good(L_{d,n})$.
Because $\good(L_{d,n})$ is connected and $G$ is closed and not empty, 
its complement $H$ must be empty to be closed.  Hence, 
$G = \good(L_{d,n})$.

To finish the proof, recall that Lemma \ref{lem:closeG} says 
that the Zariski closure of
$\good(L_{d,n})$ is $L_{d,n}$.  
This closure must be contained in any variety, such as $V$, that contains
$\good(L_{d,n})$.
Since we also have $V\subseteq L_{d,n}$, equality holds and 
we get irreducibility.
\end{proof}

And now we can complete the proof of our theorem:

\begin{proof}[Proof of Theorem~\ref{thm:Lvariety}]
$\LN$ can be seen to be a variety by pulling back the 
defining equations of the variety $M_{d,n}$ through $s$.
Dimension is  Lemma~\ref{lem:dim}.
Irreducibility is Proposition~\ref{prop:irr}.
\end{proof}

\section{Automorphisms of $M_{d,n}$}

\begin{definition}\label{def:linear-automorphism}
A \defn{linear automorphism} of a variety 
$V$ in $\CC^N$ is 
a non-singular linear transform on $\CC^N$
(that is, a non-singular $N\times N$ complex matrix $\A$) that bijectively
maps $V$ to itself.\footnote{In our setting, $V$ will always be a cone, so 
linear isomorphisms (as opposed to affine ones) are natural.}
\end{definition}

\begin{definition}\label{def:gen-perm}
An $N\times N$ matrix $\P$ is a
\defn{permutation} if each row and 
column has a single non-zero entry, and this entry
is $1$.
A matrix $\P'=\D\P$, 
where $\D$ is diagonal and invertible, is a 
\defn{generalized permutation}.
Each row and 
column has exactly one non-zero entry.  
A generalized permutation has \defn{uniform scale} 
if it is a scalar multiple of a permutation matrix.  
\end{definition}

\begin{definition}
A generalized
permutation acting on an edge set
is 
\defn{induced by a vertex relabeling}
when it has the same
non-zero pattern as an edge permutation that arises from a vertex relabeling.
\end{definition}

We  now present the following slight generalization  of
\cite[Lemma 2.4]{BK1}. Here we deal 
with generalized permutations instead of 
permutations, but the same proof applies.
\begin{theorem}[{\cite[Lemma 2.4]{BK1}}]\label{thm:bk-linear-s}
Suppose that $\A$ is a generalized permutation
that is a linear automorphism of $M_{d,n}$.  Then 
$\A$ is induced by a vertex relabeling.
\end{theorem}

The following material will help us slightly strengthen
Theorem~\ref{thm:bk-linear-s}, and will also be used
later in Section~\ref{sec:maps}.

First we define the 
combinatorial notion of infinitesimally dependent and independent
sets of edges in $d$ dimensions.

\begin{definition}
\label{def:maps}
Let $d$ be some fixed dimension and $n$ a number of vertices.
Let $E:= \{E_1,\ldots, E_k\}$ be  an ordered subset of the $N$  edges.
The ordering on the edges of $E$
fixes  an association between each edge in $E$ 
and a coordinate axis of $\CC^{k}$. 
Let $m_E(\p)$
be the map from $d$-dimensional 
configuration space to $\CC^{k}$
measuring the squared lengths of the edges
of $E$. 

We denote by $\pi_{\bar{E}}$
the linear map  from $\CC^N$
to $\CC^{k}$ 
that forgets the edges not in $E$, and is consistent with the 
ordering of $E$.
Specifically, we have an association between each edge of
$K_n$ and an index in $\{1,\dots,N\}$, and thus we can think of each 
$E_i$ as simply its index in $\{1,\dots,N\}$. Then,
$\pi_{\bar{E}}$ is defined by the conditions:
$\pi_{\bar{E}}(e_j) = 0$ when $j\in \bar{E}$
and
$\pi_{\bar{E}}(e_j) = e'_i$ when $E_i=j$,
where 
$\{e_1, \ldots, e_N\}$ denotes  the coordinate basis for $\CC^N$
and
$\{e'_1, \ldots, e'_k\}$ denotes  the coordinate basis for $\CC^k$.
We call $\pi_{\bar{E}}$ an \defn{edge forgetting map}.

With this notation, 
the map $m_E(\cdot)$ is simply the composition of the complex measurement map $m(\cdot)$ 
and $\pi_{\bar{E}}$.

\end{definition}

\begin{definition}
\label{def:ind}
We say the 
an edge set  $E$  is \defn{infinitesimally independent} in $d$
dimensions if there exists a  complex configuration 
$\p$ in $\CC^d$, where we can
\emph{differentially}
vary each of the $|E|$ 
squared lengths independently, by
appropriately differentially varying our  configuration $\p$. 

Formally, this means that the 
image 
of the differential of $m_E(\cdot)$ at 
$\p$ is $|E|$-dimensional.
This exactly
coincides with the notion of infinitesimal independence from graph 
rigidity theory~\cite{L70}.

We call such a configuration $\p$, \defn{$E$-regular}.
Every configuration in some appropriate  
neighborhood of an $E$-regular point is also
$E$-regular (by semi-continuity).
 This neighborhood must include configurations
with full affine spans and no coincident points

For any configuration $\p$ with full affine span, 
$m(\p)$ is smooth (Theorem~\ref{thm:Mvariety}).
Thus for any $E$-regular 
configuration $\p$, with full affine span, 
using the chain rule, the differential image of $\pi_{\bar{E}}$ at 
the point $m(\p)$ is $|E|$-dimensional.
We call such a point of $M_{d,n}$, \defn{$E$-regular}.
Such points must exist when $E$ is infinitesimally independent.

For any smooth point $\x$ of $M_{d,n}$ with 
no zero coordinates, all of its preimages under the squaring map, $s()$, are
smooth in $L_{d,n}$ (Lemma~\ref{lem:Asmooth}).
Thus for any preimage $\genericpoint$ of an $E$-regular point 
$\x$, with no zero coordinates, 
the differential image of $\pi_{\bar{E}}$ at 
the point $\genericpoint$ is $|E|$-dimensional.
(As the Jacobian of $s(\cdot)$ at $\bl$ 
is diagonal and
bijective). 
We call such a point of $L_{d,n}$, \defn{$E$-regular}.
Such points must exist when $E$ is infinitesimally independent.

An edge set that is 
not infinitesimally independent in $d$ dimensions is 
called \defn{infinitesimally dependent} in $d$ dimensions.
\end{definition}


The following is a standard result from rigidity theory
(see, e.g., \cite[Corollary 2.6.2]{GSS}).  
\begin{proposition}
\label{prop:simplex}
Let $E$ be an edge set (with all its edges
distinct).
Suppose $|E| \le \binom{d+2}{2}$ and $E$ is 
infinitesimally dependent
in $d$ dimensions.
Then $|E| = \binom{d+2}{2}$ and
$E$ consists of the edges of a  $K_{d+2}$ subgraph (in some order).
\end{proposition}
\begin{proof}[Proof Sketch]
Assume, w.l.o.g., that $E$ is infinitesimally
dependent and inclusion-wise minimal
with this property.  
If $E$ does not consist of the edges of a $K_{d+2}$ subgraph, then it has a vertex $v$ 
of degree at most $d$.  
Let $\p$ be in general affine position.  This means, in particular, that $\p_v$ is
not in the affine span of its neighbors.  Hence, the $\le d$ squared lengths of each edge in 
edge set $E'$ incident on $v$ can be 
differentially varied independently (by exercising the $d$ degrees of freedom in $\p_v$).
Thus the edges of $E'$ can be removed from $E$ leaving the remainder,
$E\setminus E'$, still infinitesimally
dependent. This  contradicts the assumed minimality of $E$.
\end{proof}

\begin{lemma}\label{lem:1d-Mdn}
Any linear  automorphism $\A$ of $M_{d,n}$ is 
a linear  automorphism of $M_{1,n}$.
\end{lemma}
\begin{proof}
The singular set of $M_{d,n}$ is $M_{d,n-1}$ by Theorem 
\ref{thm:Mvariety}.  Thus, 
from Theorem~\ref{thm:Tmap}, $\A$ must be a linear 
automorphism of $M_{d-1,n}$.  We then see, by induction, 
that $\A$ is also a linear automorphism of $M_{1,n}$.
\end{proof}
In fact, this kind of induction has been recently used
to greatly 
strengthen Boutin and Kemper's 
unique reconstructability 
result~\cite{BK1} to apply
to a much larger class of graphs than just
the complete graphs~\cite{gugr}.

\begin{lemma}\label{lem:voldet-scale}
Let $m_{12}$, $m_{13}$ and $m_{23}$ be the squared
edge lengths of a $1$-dimensional triangle, and
suppose that $s_{12}$, $s_{13}$ and $s_{23}$
are scalars such that the simplicial 
volume determinant 
\[
\begin{split}
\det
\begin{pmatrix}
2m_{13}& (m_{13} + m_{23} - m_{12})\\
(m_{13} + m_{23} - m_{12})& 2m_{23}\\
\end{pmatrix} = & \\  & 2 (m_{12}m_{13} + m_{12}m_{23} + m_{13}m_{23})
- (m^2_{12} + m^2_{13} + m^2_{23})
\end{split}
\]
(see Section~\ref{sec:varieties}) is mapped to a 
multiple of itself under the scaling $m_{ij}\mapsto s_{ij}m_{ij}$.  Then the $s_{ij}$ are all equal.
\end{lemma}
\begin{proof}
The hypothesis means that the desired statement holds for any specialization 
of the $m_{ij}$.  Consider the case where $m_{23} = 0$.
The presence of the monomials $m_{12}^2$ and 
$m_{12}m_{13}$ then imply that $s_{12}^2 = s_{12}s_{13}$, that is, $s_{12} = s_{13}$.  Continuing the same way, 
we see that $s_{12} = s_{23}$. 
\end{proof}

Now we can state the following slight  strengthening of
Theorem~\ref{thm:bk-linear-s}.

\begin{theorem}\label{thm:bk-linear}
Suppose that $\A$ is a generalized permutation
that is a linear automorphism of $M_{d,n}$.  Then 
$\A$ is induced by a vertex relabeling and has 
uniform scale.
\end{theorem}
\begin{proof}
Theorem~\ref{thm:bk-linear-s}
tells us that 
$\A$ is induced by a vertex relabeling.
Next we need to prove uniform scale. From 
Lemma~\ref{lem:1d-Mdn}, we can look at 
$A$ as an automorphism of $M_{1,n}$

Let $\pi_{\bar{K}}$ be an edge forgetting map 
that ignores all of the edges in the complement
of an edge set $K$, consisting of the edges of a fixed
triangle. Under any ordering of the edges of $K$, we have 
$\pi_{\bar{K}} (M_{1,n}) = M_{1,3}$.
(which is cut out from $\CC^3$ by the simplicial 
volume determinant as in Lemma~\ref{lem:voldet-scale}).

We know that we can factor $\A$ into 
$\D\P$, where $\D$ is diagonal and 
$\P$ is a 
permutation induced by a vertex relabeling.
Since a vertex relabeling is a linear automorphism of $M_{1,n}$, then so too is $\D$.

Since $\D$ is diagonal, and $\pi_{\bar{K}}$ is an
edge forgetting map, then $\pi_{\bar{K}} \D = \D' \pi_{\bar{K}}$ 
for an appropriate $3\times 3$ diagonal scaling matrix $\D'$, 
 making 
$\D'$ an automorphism of $M_{1,3}$. So it has to 
send the simplicial volume determinant to a multiple of itself.
This is the situation of Lemma~\ref{lem:voldet-scale},
and we conclude that the scaling on each triangle is uniform.

That $\A$ has a uniform scale then follows from applying 
the above argument repeatedly to overlapping triangles
until we have determined the scale on every edge.
\end{proof}

\section{Linear maps from $L_{d,n}$ to $\CC^D$}
\label{sec:maps}

In this section, 
which forms the technical heart of this paper,
we will study how linear projections
act on $L_{d,n}$.

Let $d \ge 1$.
Recall that $D:= \binom{d+2}{2}$.
In this section,
$\E$ will be 
a $D\times \edgecard$ matrix
representing a rank $r$ linear map from $\LN$ to
$\CC^D$, where 
$r$ is some number $\le D$ . 
Our goal is to study linear maps where the 
dimension of the image is strictly less than $r$. 
In particular this will occur when $\E(\LN) = L_{d,d+2}$.

\begin{definition}
We say that $\E$ has 
$K_{d+2}$ support 
if it depends only on measurements supported over the $D$ 
edges corresponding to a $K_{d+2}$ subgraph of $K_n$. 
Specifically,  
all the columns of the matrix 
$\E$ are zero, except for at most $D$ of
them, and these non-zero columns index edges 
contained within a single $K_{d+2}$.
\end{definition}

The main result of this section is:
\begin{theorem}
\label{thm:linImage}
Let $\E$ be  a 
$D\times \edgecard$ matrix with rank $r$.
Suppose that the image $\E(L_{d,n})$,
a constructible set,  
is not of dimension $r$.
Then  $r=D$ and $\E$ has $K_{d+2}$ support.
\end{theorem}
\begin{remark}
Theorem~\ref{thm:linImage} does not hold when 
$\LN$ is replaced by $M_{d,n}$. As described in the introduction, the linear automorphism group of
$\CS^{n-1}_d$ is quite large,  and thus provides 
automorphisms $\A$ of $M_{d,n}$ that have
dense support. Thus, even if some $\E$ has $K_{d+2}$ support
the composite map $\E\A$ would not, and it 
could still have a low-dimensional image. 
\end{remark}

The proof relies (crucially) on the more technical, linear-algebraic 
Proposition~\ref{prop:inf}, proved below.  The idea
leading to it is as follows.

If a
point $\bl$ is smooth in $L_{d,n}$ then so is 
any $\bl'$ obtained by negating various coordinates of $\bl$.
Thus, the collection of complex analytic 
tangent spaces to $L_{d,n}$,
$T_{\bl}L_{d,n}$,
at $\bl$ and its orbit under coordinate negations gives us 
an arrangement $\mathcal{T}$ 
of $2^N$ linear spaces (related through
coordinate negation).
Any $\E$ meeting the 
hypothesis of Theorem~\ref{thm:linImage} necessarily drops 
rank on every subspace in $\mathcal{T}$.  
This would not be possible if $\E$
or the collection of tangent spaces
$T_{\bl}L_{d,n}$,
were sufficiently general.  
On the other hand, we know that the geometry of our situation is 
 special enough that when $\E$ has rank $D$
and
 $K_{d+2}$ support, then $\E$ does drop rank on each of the $T_{\bl}L_{d_n}$.  
Proposition~\ref{prop:inf} asserts that
this is the \textit{only} possibility.
This proof relies on the negation-based
symmetry of $L_{d,n}$ and
on
the fact that $K_{d+2}$ is the only graph on $D$ or 
fewer edges that is  infinitesimally dependent
(Proposition~\ref{prop:simplex}). 

First we present the proof of Theorem~\ref{thm:linImage},
which effectively reduces our problem to the linear 
situation covered in Proposition~\ref{prop:inf}.

\begin{proof}[Proof of Theorem~\ref{thm:linImage}]
Clearly, the image of the map must be contained in an $r$-dimensional
linear space spanned by the columns of $\E$.
Suppose that either $r<D$, or $\E$ does not have $K_{d+2}$ support.
Then, from Proposition~\ref{prop:inf} below,
there must be a smooth point
$\bl'$
such that
$\Dim(\E(T_{\bl'}L_{d,n}))=r$. 
Then, from the Local Submersion 
Theorem for smooth maps~\cite[page 20]{gp},
the map must be locally surjective onto the 
$r$-dimensional linear space. Thus the image 
(a constructible set) cannot have smaller 
dimension.
\end{proof}

We are now ready to state the key technical result in this 
section.
\begin{proposition}\label{prop:inf}
Let $\E$ be  a 
$D\times \edgecard$ matrix with rank $r$.
Suppose that either $r<D$ or $\E$ does not have $K_{d+2}$ support.
Then 
there is a smooth 
point
$\bl' \in L_{d,n}$ with 
with  the property
that 
$\Dim(\E(T_{\bl'}L_{d,n})) =r$.
\end{proposition}

\subsection{Proof of Proposition~\ref{prop:inf}}
The rest of the section is occupied with the proof, 
which we break down into steps.
We use a technical lemma about coordinate 
negation and determinants that is relegated to it own Section~\ref{sec:det}.

\begin{definition}
\label{def:sf}
A \defn{sign flip matrix} $\S$ is a diagonal matrix 
with $\pm 1$ on the diagonal.  A \defn{coordinate flip}
of a point or subspace it its image under a sign 
flip matrix.
\end{definition}

\begin{definition}
\label{def:tans}
Let $\bm$ be a smooth  point in $M_{d,n}$, 
and $T_{\bm}M_{d,n}$ be its complex analytic  tangent.
We can describe 
$T_{\bm}M_{d,n}$ by a
$(dn-C)\times \edgecard$
complex matrix $\T_\bm$. (The row ordering is not relevant).

Referring back to Definition~\ref{def:ind}, 
if $E$ is an infinitesimally independent edge set, then
the columns of $\T_\bm$ 
corresponding to $E$,
at an $E$-regular point of $M_{d,n}$,
are linearly independent.
The same is true 
of the matrix $\T_\bl$ that expresses
the tangent space 
$T_{\bl}L_{d,n}$ at an $E$-regular  point $\bl$ of 
$L_{d,n}$. 
Such points must exist when $E$ is infinitesimally independent.
\end{definition}

The first step is to restrict to an interesting range
of $n$.
\begin{lemma}\label{lem:flips:n-D-range}
Proposition~\ref{prop:inf} holds when $n<d+2$.
\end{lemma}
\begin{proof}
When  $n \le d+1$, 
$T_{\bl}L_{d,n}$ is equal to the full embedding space, and 
thus $\Dim(\E(T_{\bl}L_{d,n})) =r$.  Proposition~\ref{prop:inf}
is then trivial in this case.
\end{proof}
Thus, from now on, we may assume that 
$n\ge d+2$.

Let $\T$ be a $(dn - C)\times N$
matrix with rows spanning the tangent space $T_{\bl}L_{d,n}$
at some smooth point $\genericpoint$.
The complex analytic  tangent space at a
smooth point of a variety with pure dimension
has the same dimension
as the variety, which explains the shape of $\T$.

\paragraph{Block form and column basis}
Each column of $\E$ and $\T$ corresponds to an edge in $K_n$. We are
going to make use of edge-permuted versions of these matrices that have
particular block structures. To this end, we are now going to look at
the columns of $\E$ and determine which subsets can form a basis,
$\E_2$, of a linear space of dimension $r$.  So we permute and then
partition the columns of $\E$ into a block form
\[
	\begin{pmatrix}
    	\E_1 & \E_2
    \end{pmatrix}.
\]
where $\E_1$ is $D\times (N-r)$
and $\E_2$ is $D\times r$.
We define a column basis, $\E_2$ of $\E$, to be \defn{good} when $r=D$ and 
the columns of $\E_2$ correspond to the edges of a $K_{d+2}$. 
Any other column basis $\E_2$ will be called \defn{bad}.
We denote by $E_2$ the edge set corresponding to the columns 
of $\E_2$.

Suppose that $\E$ has $K_{d+2}$ support 
and $r=D$. 
then the $r$ columns of $\E$ corresponding 
to the edges of this $K_{d+2}$ 
must form the only column basis of $\E$. 
Moreover, it is good. 

\begin{lemma}\label{lem:flips:bad-basis}
If  $\E$ does not 
have $K_{d+2}$ support or $r < D$,
then there is a bad column basis for $\E$. 
\end{lemma}
\begin{proof}
If
$r < D$, then by definition, 
no column basis can be good.  From now on, then, 
assume that $r=D$.

If $\E$ is supported on only $D$ columns, there is a unique column 
basis $\E_2$.  
Thus in this case, non-$K_{d+2}$ support for $\E$ will imply that the
unique column basis is bad.

Suppose instead there are more than $D$ 
non-zero columns of $\E$.
Thus, starting from, say, a good basis $\E_2$, we 
can exchange a non-zero column of $\E_1$ with an appropriate one 
from $\E_2$ to obtain another basis which 
is bad: removing an edge from a $K_{d+2}$
and replacing it with any other edge results in a 
graph that cannot be a $K_{d+2}$ (it has more vertices).
\end{proof}
\begin{remark}\label{rem:flips:bad-basis-iff}
In light of the paragraph preceding this lemma, 
Lemma~\ref{lem:flips:bad-basis} can be made into 
an ``if and only if'' statement.
\end{remark}

Going back to $\T$ and applying the same column used obtain 
$	\begin{pmatrix}
    	\E_1 & \E_2
    \end{pmatrix}$,
we get a block form
\ba
    \begin{pmatrix}
        \T_1  & \T_2 
    \end{pmatrix}
\ea
where 
$\T_1$ is $(dn-C)\times (N-r)$ and
$\T_2$ is $(dn-C)\times r$.

\begin{lemma}\label{lem:flips:T2-cols}
Assuming that $\E_2$ is a bad basis of $\E$ and $\bl$ is $E_2$-regular, 
the matrix $\T_2$ has rank $r$ (and 
in particular has linearly independent columns) 
\end{lemma}
\begin{proof}
Since $(\E_1,\E_2)$ arises from a bad basis, and we have
only applied column permutations, the columns
of $\T_2$ corresponds to a subgraph $G$ of $K_n$
with at most
$D$ edges which is not $K_{d+2}$.  
Proposition~\ref{prop:simplex} tells us that the edges of 
$G$ are infinitesimally independent.
So, by 
$E_2$-regularity of $\bl$, these columns of $\T$ are 
linearly independent (Definition~\ref{def:tans}).
\end{proof}

\paragraph{Row rank}

\begin{lemma}\label{lem:flips:T1-T2-rows}
Assuming that $\E_2$ is a bad basis of $\E$ and $\bl$ is $E_2$-regular.
Then the block matrix 
\(
	\begin{pmatrix}
    	\T_1 & \T_2
    \end{pmatrix}
\) contains $r$ rows,
\(
	\begin{pmatrix}
    	\T'_1 & \T'_2
    \end{pmatrix}
\), such that 
$\T'_2$ forms a non-singular matrix.
\end{lemma}
\begin{proof}
Since we have a bad basis, 
from Lemma~\ref{lem:flips:T2-cols},
$\T_2$ has $r$ linearly independent columns
and thus $r$ linearly independent rows.
We can select any set of rows corresponding to a row basis of $\T_2$.
\end{proof}

Similarly, we have
\begin{lemma}\label{lem:flips:E1-E2-rows}
Let $\E_2$ be a column basis for $\E$.
Then the block matrix 
\(
	\begin{pmatrix}
    	\E_1 & \E_2
    \end{pmatrix}
\) 
contains $r$ rows,
\(
	\begin{pmatrix}
    	\E'_1 & \E'_2
    \end{pmatrix}
\), such that 
$\E'_2$ forms a non-singular matrix.
\end{lemma}

Next, we derive an implication of $\E$ dropping rank on
the tangent space.

\begin{lemma}
\label{lem:flips-row-rank}
Suppose there is a 
smooth point
$\bl\in L_{d,n}$ such that $\bl$ and 
all of its coordinate flips $\bl'$  have the property
that 
$\Dim(\E(T_{\bl'}L_{d,n})) <r$.
Let $\E_2$ be a bad basis for $\E$.
Let $\S_1$ be any  
any $(N-r)\times (N-r)$ sign flip matrix, and 
$\S_2$, any $r\times r$ sign flip matrix.

Then
the $r\times r$ matrix 
$\Z:= 
\E'_1 \S_1 \trans{\T'_1}
+ \E'_2 \S_2 \trans{\T'_2}
$ is singular.
\end{lemma}
\begin{proof}
Let $\S$ be the $N\times N$ be the
sign flip matrix with the $\S_i$ as its
diagonals. Let $\bl'$ be the point obtained from $\bl$ under the sign flips of $\S$. Because $L_{d,n}$ is symmetric under coordinate 
negations, then $T_{\bl'}L_{d,n}$ is spanned by the columns of
$\S \trans{\T}$.
Then we have
$\Dim(\E(T_{\bl'}L_{d,n})) =
\rank(\E \S \trans{\T}) = 
\rank(
\E_1 \S_1 \trans{\T_1} 
+ \E_2 \S_2 \trans{\T_2}) 
\ge
\rank(
\E'_1 \S_1 \trans{\T'_1} 
+ \E'_2 \S_2 \trans{\T'_2})$.

If for some $\S$, the matrix $\Z$ were non-singular,
then we would have a certificate that
$\E$ does not drop rank on that coordinate flip of the 
tangent space, in contradiction to the hypothesis 
on $\Dim(\E(T_{\bl'}L_{d,n}))$.
\end{proof}
\begin{remark}
The rank of $\Z$ may change as the $\S_i$ do, but it 
cannot rise to $r$.
\end{remark}

\paragraph{Conclusion of the proof}
Assume that 
$\E$ does not 
have $K_{d+2}$ support or $r < D$.
From Lemma~\ref{lem:flips:bad-basis},
there is a bad column basis $\E_2$ for $\E$. 
From Lemma~\ref{lem:flips:T1-T2-rows}, for an $E_2$-regular $\bl$, 
$\T'_2$ is a non-singular matrix.

Suppose that at this 
$\bl$, we had for 
all of its coordinate flips $\bl'$,   the property
that 
$\Dim(\E(T_{\bl'}L_{d,n})) <r$.
Then
from Lemma~\ref{lem:flips-row-rank},
for any choice of  $\S_2$, 
the matrix 
$\Z$ would be singular. 
Since $\E_2$ is a basis, 
$\E'_2$ is non-singular matrix
(Lemma~\ref{lem:flips:E1-E2-rows}),
thus
$\Z':= \S_2 {\E'_2}^{-1} \Z = 
\S_2 ({\E'_2}^{-1} \E'_1 \S_1 \trans{\T'_1})
+  \trans{\T'_2}$ 
would be singular for any choice of $\S_2$.
Thus, Lemma~\ref{lem:nonsing} on determinants and sign flips (below)
would apply to $\Z'$, and we would
conclude that $\T'_2$ is singular. 

From the resulting contradiction, we can deduce that for an $E_2$-regular point 
$\bl$, one of its coordinate flips $\bl'$ must have 
$\Dim(\E(T_{\bl'}L_{d,n})) = r$. By regularity, $\bl$ is smooth, and so too is any coordinate flip such as $\bl'$.
\qed

\subsection{Determinants and flips}
\label{sec:det}
In this section, we will establish 
a technical
lemma about determinants and sign flips.

\begin{lemma}\label{lem:nonsing}
Suppose that $\Z = \S\X + \Y$ is an $r\times r$ matrix 
and
$\det(\Z) = 0$ for all choices of sign flips, $\S$. Then 
$\det(\Y) = 0$.
\end{lemma}
\begin{proof}
Multilinearity of the determinant allows us to 
express $\det(\Z)$ 
as $\det(\Z')+\det(\Z'')$,
where $\Z'$ is the matrix $\Z$
with its first row replaced
by the first row of $\S\X$, and
where $\Z''$ is the matrix $\Z$
with its first row replaced
by the first row of $\Y$.
We can likewise expand out each
of $\det(\Z')$ and $\det(\Z'')$ 
by splitting their second rows.
Applying this decomposition recursively  
we ultimately get:
\[
	\det(\S\X + \Y) = \sum_{I\subseteq [r]} \det(\Z^\S_I)
\]
where $[r] = \{1,2,\dots,r\}$,
and $\Z^\S_I$ is the matrix that has the rows indexed by $I$ from $\S\X$ and the rest 
from $\Y$.

Now  sum the above over the $2^r$ choices of $\S$ and rearrange
\[	
	\sum_{\S} \det(\S\X + \Y) = \sum_{\S}\sum_{I\subseteq [r]} \det(\Z^\S_I) = 
    \sum_{I\subseteq [r]} \overbrace{\sum_{\S} \det(\Z^\S_I)}^{\star}
\]
For fixed $I$, each $\det(\Z^\S_I) =  (-1)^{\sigma(\S,I)}\det(\Z^{\I}_I)$,
where $\sigma(\S,I)$ is the number of rows corresponding to $I$ where 
$\S$ has a diagonal entry of $-1$.  Thus, for each $I$, ($\star$) is
\[
	2^{r-|I|}\cdot\left(\sum_{k=0}^{|I|}\binom{|I|}{k}(-1)^k\right)
    \cdot\det(\Z^{\I}_I)
\]
(The power of two factor accounts for all of the
sign choices in $\S$ over the complement of $I$.)
The coefficient of $\det(\Z^{\I}_I)$
equals $2^{r}$ when $I$ is empty. Otherwise 
it is zero since the inner term is simply the 
binomial expansion of
$(1-1)^{|I|}$. Thus,
\[
	\sum_{\S} \det(\S\X + \Y) =  2^r \det(\Y)
\]
Since this sum vanishes by hypothesis, we get $\det(\Y) = 0$.
\end{proof}

\section{Automorphisms of $L_{d,n}$}\label{sec:Ldn-automorphisms}

In this section we will characterize the linear automorphisms of
$L_{d,n}$ for all $d$ and $n$.
One key feature will be that we are no longer 
restricted to the case of edge permutations.

We will need to consider a few distinct cases for $d$ and $n$.


\begin{definition}\label{def:signed-perm}
Set $N := \binom{n}{2}$ and identify the rows and 
columns of an $N\times N$ matrix with the 
edges of $K_n$.

A \defn{signed permutation} is an $N\times N$
matrix $\P'$ that is the product $\S\P$ of a 
sign flip matrix $\S$ and a permutation 
matrix $\P$.

A signed permutation $\P':=\S\P$ is 
\defn{induced by a vertex relabeling} if 
$\P$ is induced by a vertex relabeling of $K_n$.
\end{definition}

\subsection{Automorphisms of $L_{d,n}$, $n\ge d+3$}
\label{sec:bign}

Let $d\ge 1$.
This section will be concerned with $L_{d,n}$ where $n$ is larger than
the minimal value, $d+2$.

\begin{theorem}\label{thm:no-regges}
Let $n \ge d+3$.  Then any linear automorphism 
$\A$ of $L_{d,n}$ of is a scalar multiple
of a signed permutation that is induced by a 
vertex relabeling.
\end{theorem}

The plan is to use 
machinery from Section~\ref{sec:maps}
to show that the automorphism must be in the form 
of a generalized edge permutation. We will then be able to
switch over 
to the $M_{d,n}$ setting, where we can 
apply Theorem~\ref{thm:bk-linear}.

\begin{definition}\label{def:support-map}
Let $\A$ be an $N\times N$ matrix.  We identify the 
rows and columns of $\A$ with the edges 
of $K_n$.  This induces a
map $\tau_\A$ from subgraphs of $K_n$ to 
subgraphs of $K_n$ by mapping the subgraph 
associated with a collection of rows 
to the column support of this sub-matrix.
\end{definition}
\begin{lemma}\label{lem:tet-a-tet}
Let 
$n\ge d+2$ and suppose that $\A$ is a 
linear automorphism of $L_{d,n}$.
 Then the associated combinatorial map
$\tau_\A$ induces a permutation on  $K_{d+2}$ 
subgraphs of $K_n$.
\end{lemma}
\begin{proof}
If $\E$ is any $D \times N$ matrix of 
rank $D$, with $\E(L_{d,n}) \subset  L_{d,d+2}$,
then
the map $\E\A$  also has these properties.
Thus, by Theorem~\ref{thm:linImage} both $\E$ and 
$\E\A$ have $K_{d+2}$ support. 
There is such 
an $\E$ for each $K_{d+2}$ subgraph: simply take the matrix
of the edge forgetting map $\pi_{\bar{K}}$, where $K$ is
an edge set comprising the edges of 
this $K_{d+2}$.
This situation is only possible 
if $\tau_\A(T)$ maps each $K_{d+2}$ 
subgraph $T$ to another $K_{d+2}$  subgraph.  

If the map on $K_{d+2}$  subgraphs induced by $\tau_\A$ is not 
injective, then the matrix  $\A$ would have 
more than $D$  rows supported by only $D$ columns, and 
thus $\A$ would be singular.  Since 
$\A$ is a linear automorphism of $L_{d,n}$ it has to 
be invertible, and the resulting contradiction 
completes the proof.
\end{proof}

This lets us prove the following.
\begin{lemma}\label{lem:gen-perm}
Let $n\ge d+3$ and let $\A$ be a linear automorphism
of $L_{d,n}$.  Then $\A$ is a generalized permutation.
\end{lemma}
\begin{proof}
Suppose, w.l.o.g., that the row corresponding to 
the edge $e:=\{1,2\}$ has two non-zero entries corresponding 
to edges $\{i,j\}$ and $\{k,\ell \}$.  
By  Lemma~\ref{lem:tet-a-tet}, 
any $K_{d+2}$ 
 subgraph 
$T$ containing the edge $e$ must be mapped by $\tau_\A$ to a $K_{d+2}$  
subgraph $T'$ that contains the vertex set 
$X:=\{i,j\}\cup \{k,\ell\}$.

Since $|X|\ge 3$ there are at most 
$\binom{n-3}{d-1}$ 
choices 
for $T'$.  Meanwhile, there are $\binom{n-2}{d}$ choices for 
$T$.  Since $n\ge d+3$, we have $\binom{n-2}{d} > \binom{n-3}{d-1}$, contradicting the permutation
of $K_{d+2}$ subgraphs
guaranteed by 
Lemma~\ref{lem:tet-a-tet}.

Thus each row of $\A$ can have at most one non-zero entry.
As a non-singular matrix, this makes $\A$ a generalized
permutation.
\end{proof}
At this point, we want to move back 
to the setting of $M_{d,n}$, which we do with this next 
result.
\begin{lemma}\label{lem:L-to-M}
Let $\A := \D\P$ be a generalized permutation, where
$\D$ is an invertible
diagonal matrix and $\P$ is a permutation matrix.
If $\A$ is a linear automorphism of $L_{d,n}$ 
then $\D^2\P$ is a
linear automorphism of $M_{d,n}$.
\end{lemma}
\begin{proof}
Let $\bl^2$ denote the vector of coordinate-wise square
of a vector $\bl \in \CC^N$; in this proof squares of vectors
are coordinate-wise.  Now we check that
\ba
	\bl^2 \in M_{d,n} & \Rightarrow &  \\
    \bl \in \LN & \Rightarrow & \\
    \D\P\bl\in \LN & \Rightarrow & \text{($\A$ is an automorphism)} \\
    (\D\P\bl)^2\in \ M_{d,n} & \Rightarrow & \\
    \D^2(\P\bl)^2 \in M_{d,n} & \Rightarrow &  \text{($\D$ is diagonal)} \\
    (\D^2\P)\bl^2 \in M_{d,n}& & \text{($\P$ is a permutation)}
\ea
\end{proof}

\begin{proof}[Proof of  Theorem~\ref{thm:no-regges}]
From Lemma~\ref{lem:gen-perm},
any linear automorphism $\A$ of 
$L_{d,n}$ with $n\ge d+3$ is a generalized permutation
$\A=\D\P$.
Lemma~\ref{lem:L-to-M} implies that $\A$ 
gives rise to a generalized edge permutation $\D^2\P$
that is a linear automorphism of $M_{d,n}$. Theorem 
\ref{thm:bk-linear} then tells us that 
$\D^2\P=s^2 \P$ has uniform scale and also is induced
by a vertex relabeling.  Finally
$\A$ is then a scalar multiple of a signed permutation 
(Lemma~\ref{lem:L-to-M} ``forgets'' the signs) as 
required.
\end{proof}

\subsection{Automorphisms of $L_{d,d+2}$, with $d\ge 3$}
\label{sec:min3}
Our next case is when $n$ is minimal, but we will only deal with
the case of $d \ge 3$.

\begin{theorem}\label{thm:no-regges3}
Let $d \ge 3$.  
Then any linear automorphism 
$\A$ of $L_{d,d+2}$  is a scalar multiple
of a signed permutation that is induced by a 
vertex relabeling.
\end{theorem}

The plan is to use some of the structure of the singular locus
of $L_{d,d+2}$ to reduce our problem to that of $L_{d-1,d+2}$.
Then we can directly apply 
Theorem~\ref{thm:no-regges}.

\begin{lemma}
\label{lem:singComp}
Let $d\ge 3$. 
$L_{d-1,d+2}$ is an irreducible subvariety of $\sing(L_{d,d+2})$.
\end{lemma}
\begin{proof}
Looking first at the squared measurement variety, 
from Theorem~\ref{thm:Mvariety}, we know that 
$\sing(M_{d,d+2})=M_{d-1,d+2}$. 

Let $Z$ be the locus of $\CC^N$ where at least one coordinate vanishes, and 
let  $S:= L_{d-1,d+2} - Z$.
Thus from Lemma~\ref{lem:Asmooth}, the points in $S$, 
are (algebraically) singular in $L_{d,d+2}$.
So $S$ is contained in  $\sing(L_{d,d+2})$.

From Theorem~\ref{thm:Lvariety}, when $d\ge 3$, we have
$L_{d-1,d+2}$ 
is irreducible.
The set $S$ is obtained from 
$L_{d-1,d+2}$ by removing a strict subvariety, which must
be of lower dimension due to irreducibility. Thus $S$ is a full-dimensional 
constructible subset of the irreducible
$L_{d-1,d+2}$.
Thus the Zariski closure of $S$ is $L_{d-1,d+2}$.

Since $\sing(L_{d,d+2})$  is an algebraic variety, it must
contain the Zariski closure of $S$ which is $L_{d-1,d+2}$.
\end{proof}

\begin{lemma}
\label{lem:singBigSpan}
$L_{d-1,d+2}$ has a full-dimensional affine span.
\end{lemma}
\begin{proof}
Since $L_{d-1,d+2}$ contains $L_{1,d+2}$, we just need to show that
this smaller variety has a full-dimensional affine span.

For a fixed $i$, 
let us look at configuration $\p$ of $d+2$ points with  $\p_i$
placed at $1$ and the rest of the points
placed at the origin. Then $\bl:=l(\p)$ has all zero coordinates
except for the $d+1$ edges connecting $\p_i$ to the other points.
Under the symmetry of $L_{1,d+2}$ under sign negation, we can find
points in $L_{1,d+2}$ with the signs of the $\bl$ flipped at will.
Thus using affine combinations of these flipped points 
we can produce a point on the $l_{ij}$ axis, for any $j$.
Iterating over the $i$ gives us our result. 
\end{proof}

Now we wish to explore the decomposition of 
$\sing(L_{d,d+2})$ 
into
its irreducible components.

For each $ij$, 
Let $Z_{ij}$ be the subvariety of 
$\sing(L_{d,d+2})$ with a zero-valued $ij$th coordinate.
As discussed above in
Lemma~\ref{lem:Asmooth} any singular point that is not contained in
$L_{d-1,d+2}$ 
must have at least one zero coordinate (in order to be
in the ``bad locus'' described there). 
Thus we can 
write 
$\sing(L_{d,d+2})$ as the union of 
$L_{d-1,d+2}$  and the $Z_{ij}$.

For $d \ge 3$, $L_{d-1,d+2}$ is irreducible,
and thus from Lemma~\ref{lem:comp} (applied to the union
of components of $\sing(L_{d,d+2})$) it
must be fully contained in at least one component $C$ of 
$\sing(L_{d,d+2})$. And, again from
from Lemma~\ref{lem:comp} (applied to the union of 
$L_{d-1,d+2}$ and the $Z_{ij}$), 
$C$ must be fully contained in 
either  $L_{d-1,d+2}$ or one of the $Z_{ij}$.
Meanwhile, $L_{d-1,d+2}$
it is not contained in
any $Z_{ij}$. 
Thus we can conclude that:
\begin{lemma}
\label{lem:bigComp}
Let $d \ge 3$.  
$L_{d-1,d+2}$ is a component of 
$\sing(L_{d,d+2})$. 
\end{lemma}

From Lemma~\ref{lem:comp}
(applied to the union of 
$L_{d-1,d+2}$ and the $Z_{ij}$),
any other component
of $\sing(L_{d,d+2})$ must be contained in one of the 
$Z_{ij}$
Thus, we can also
conclude:

\begin{lemma}
\label{lem:singSmallSpan}
Let $d \ge 3$.  
Any component of $\sing(L_{d,d+2})$ 
that is not $L_{d-1,d+2}$ cannot
have a full-dimensional affine span.
\end{lemma}

Now with this understanding of 
$\sing(L_{d,d+2})$ 
established we can move on to the automorphisms.

\begin{lemma}
\label{lem:singAuto}
Let $d \ge 3$. Any linear automorphism $\A$ of $L_{d,d+2}$ must be a 
linear automorphism of $L_{d-1,d+2}$. 
\end{lemma}
\begin{proof}
From Theorem~\ref{thm:Tmap}, $\A$ must be a linear  automorphism
of $\sing(L_{d,d+2})$. 
And from Theorem~\ref{thm:comps} must map 
components of $\sing(L_{d,d+2})$
to components of $\sing(L_{d,d+2})$. 

From Lemma~\ref{lem:bigComp}, 
$L_{d-1,d+2}$ is a component of this singular set and from 
Lemma~\ref{lem:singBigSpan} it has a full-dimensional affine span. 
Meanwhile, from Lemma~\ref{lem:singSmallSpan}, no other component
can have a full-dimensional affine span. Thus, as a bijective linear map, 
$\A$ must map $L_{d-1,d+2}$ to itself.
\end{proof}

And we can finish the proof.

\begin{proof}[Proof of Theorem~\ref{thm:no-regges3}]
The theorem now follows by combining Lemma~\ref{lem:singAuto} together
with Theorem~\ref{thm:no-regges}.
\end{proof}

\subsection{Automorphisms of $L_{2,4}$}\label{sec:auto}

The method of the previous section fails for $\LL$ as $L_{1,4}$
is reducible. In fact, the theorem itself fails 
in this case. The group of linear automorphisms 
is, in fact, larger
than expected. 

In particular, Regge~\cite{regge} (see also, Roberts \cite{regge2}) 
gave a linear map 
that always takes the
Euclidean lengths of the edges of a tetrahedral 
configuration in $\RR^2$ to those of a different tetrahedral
configuration in $\RR^2$. See Equation~(\ref{eq:reggemap}) 
in the introduction.

Below we will fully characterize the automorphism group  of $\LL$.
When we restrict
our automorphisms to have only non-negative entries, only 
the expected symmetries will remain.

\begin{definition}\label{def:real-linear-automorphism}
A linear automorphism  $\A$ of $L_{2,4}$ is \defn{real} if its matrix has
only real entries, \defn{rational} if its matrix has only rational 
entries, and \defn{non-negative} if its matrix contains only 
real and non-negative entries. 
\end{definition}

Clearly there are $24$ linear automorphism that arise by simply permuting
the $4$ vertices. There are also the $32$ linear automorphisms
that arise from optionally negating up to $5$ 
of the coordinate axes in $\CC^6$.
Combining these  gives us a discrete group of $768$ linear
automorphisms.

Because any global scale will be an automorphism, the group  of
linear automorphisms of $L_{2,4}$ is not a discrete group.  We now
define several groups that will play a role in our analysis.

\begin{definition}
Define $\Aut(L_{2,4})$ to be the linear automorphisms of $L_{2,4}$.  Let the
group $\PP \Aut(L_{2,4})$ be induced on the equivalence classes
of $\A\in \Aut(L_{2,4})$ under the relation ``$\A'$ is a complex scale of $\A$''.
We define $\PP \Aut(\sing(L_{2,4}))$ via a similar 
construction.  Importantly, 
we will see below that 
$\PP \Aut(\sing(L_{2,4}))$ 
is the automorphism group of a projective subspace arrangement
and thus is a discrete group.
Also, we have 
$\PP \Aut(L_{2,4}) < \PP \Aut(\sing(L_{2,4}))$.  Thus, all the 
``projectivized'' groups we define are discrete.

We also consider the real subgroup $\Aut_{\RR}(L_{2,4})$.  This has a 
counterpart $\PP \Aut_{\RR}(L_{2,4})$ of equivalence classes up to 
real scale, and  $\PP_{+} \Aut_{\RR}(L_{2,4})$, on equivalence classes defined 
up to \textit{positive} scale.  It is well-defined to refer to an element 
of $\PP_{+} \Aut_{\RR}(L_{2,4})$ as being non-negative, since any equivalence
class containing a non-negative $\A$ consists entirely of non-negative matrices.
\end{definition}
The main theorem of this section characterizes the linear automorphisms of $L_{2,4}$
as follows. The proof is in the next subsections.

\begin{theorem}
\label{thm:auto2}
The group $\PP \Aut(L_{2,4})$
is of order $11520=768\cdot 15$ and is 
generated by linear automorphisms of $L_{2,4}$ that are
rational.

The group $\PP_{+} \Aut_{\RR}(L_{2,4})$ is of order $23040$
and is isomorphic to the Weyl group $D_6$.  The subset of 
non-negative elements of $\PP_{+} \Aut_{\RR}(L_{2,4})$ is a 
subgroup of order $24$ and acts by relabeling the vertices of $K_4$.
\end{theorem}
\begin{remark}
The group $\PP_{+} \Aut_{\RR}(L_{2,4})$ is 
in fact generated by 
the edge permutations induced by vertex relabeling, 
sign flip matrices, and the one Regge symmetry of \eqref{eq:reggemap} 
from the introduction (see supplemental script).
\end{remark}
The rest of this section develops the proof of Theorem~\ref{thm:auto2}.

\paragraph{The Singular Locus of $\LL$}

In this section, we will study the singular locus of 
$\LL$.
This will be used for the proof of 
Theorem~\ref{thm:auto2}, which characterizes 
the linear automorphisms of $\LL$. In particular, a
linear
automorphism of a variety must also be a linear 
automorphism of its singular locus.

\begin{theorem}
The singular locus $\sing(\LL)$ consists of the union of $60$ $3$-dimensional 
linear subspaces. These subspaces can be partitioned into three types, which
we call I, II and III. 

Type I: 
There are $32$ subspaces of this type. They arise from configurations of $4$ collinear points, and together
make up $L_{1,4}$.
They are each defined by (the vanishing of) three equations of the following form:
\ba
    l_{12} - s_{13}l_{13} + s_{23}l_{23} \\
    l_{12} - s_{14}l_{14} + s_{24}l_{24} \\
    s_{13}l_{13} - s_{14}l_{14} + s_{34}l_{34} 
\ea
where each $s_{ij}$ takes on the values $\{-1,1\}$. 

Type II: 
There are $24$ subspaces of this type. They arise when one pair of vertices is
collapsed to a single point. For example, if we collapse $\p_1$ with $\p_2$, we get the 
equations:
\ba
l_{12} \\
l_{13} - s_{23}l_{23} \\
l_{14} - s_{24}l_{24}
\ea
This gives us $4$ subspaces, and we obtain this case by collapsing any of the
$6$ edges.

Type III: 
There are $4$ subspaces of this type.
They arise by setting the three edges lengths of one triangle to zero. For example:
\ba
l_{12} \\
l_{13}  \\
l_{23} 
\ea
\end{theorem}
\begin{proof}
The singular locus of a variety $V$ is defined by adding
to the ideal  $I(V)$,
the equations that express a rank-drop in the  
Jacobian matrix of  a set of equations generating $I(V)$. 

We first verify in the Magma CAS that the ideal
defined by our single simplicial volume
determinant equation is 
radical.\footnote{Magma 
does this check over the field $\QQ$, but
since $\QQ$ is a perfect field, this implies that the ideal
is also radical under any field extension~\cite[Page 169]{milne}.}
This also follows from~\cite{cmirr}.

In Magma, we calculate the Jacobian of this equation to express
the singular locus. Magma is then able to  
factor this algebraic set into  components (that are irreducible over $\QQ$), 
and in this case outputs the above decomposition. (See supplemental script.)
\end{proof}

\paragraph{Flats and intersection graph}

Theorem~\ref{thm:Tmap} tells us that any linear automorphism of $\LL$ must 
be a linear automorphism of its singular set, and so must map each of its singular 
three-dimensional subspaces to some three-dimensional singular subspace. As a linear automorphism, 
it must also preserve the intersection lattice 
of the three-dimensional singular subspace arrangement. Therefore, by finding the set 
of linear automorphisms that preserve the intersection lattice of these subspaces, 
we can constrain our search for automorphisms of $\LL$ to just that set. 
Combinatorial descriptions of an intersection lattice of a subspace arrangement 
can be constructed in many ways. Here, it suffices to consider a partial 
description that comprises the three-dimensional singular subspaces and their 
one-dimensional intersections.  

\begin{definition}
We denote by \defn{${\cal V}_3$} the set of singular three-dimensional subspaces of $\LL$. We denote by \defn{${\cal V}_1$} the set of one-dimensional subspaces created as the intersections of all pairs and triples of spaces in ${\cal V}_3$.
\end{definition}

\begin{lemma} The set of one-dimensional subspaces ${\cal V}_1$ consists of $46$ elements. These come in $3$ classes:

Type I: 
There are $6$ one-dimensional subspaces of this type. 
They are generated by vectors of the form 
\ba
\e_i
\ea
where $\e_i$ is one of the coordinate axes of $\CC^6$.

Type II: 
There are $24$ one-dimensional subspaces of this type. 
They are generated by vectors of the form 
\ba
\e_i \pm \e_j \pm \e_k \pm \e_l
\ea
where $i,j,k,l$ correspond to the four edges of
a 4-cycle.
These measurements correspond to 
collapsing two sets of two vertices that are connected by four edges.

Type III:
There are $16$ one-dimensional subspaces of this type. 
They are generated by vectors of the form 
\ba
\e_i \pm \e_j \pm \e_k \ea
where $i,j,k$ 
correspond to three edges incident to one vertex.
These measurements correspond to 
collapsing one triangle.
\end{lemma}

\begin{proof}
This follows directly from calculating the intersections of all pairs and triples of the $60$ singular subspaces of $\LL$. This has been done in the Magma CAS. 
(See supplemental script.)
\end{proof}

\begin{definition}
We define \defn{$\incidencegraph$} as the bipartite graph that has one set of vertices corresponding to the three-dimensional singular subspaces of $\LL$ (one vertex for each three-dimensional subspace), the other set of vertices corresponding to the one-dimensional intersection subspaces ${\cal V}_1$ (one vertex for each one-dimensional subspace), and an edge between vertex $i$ of the first set and vertex $j$ of the second set whenever the $i$th three-dimensional subspace includes the $j$th one-dimensional subspace.
\end{definition}

\begin{definition}
A \defn{graph automorphism} of a bipartite (two-colored) graph is a permutation $\graphautomorphism$ of the vertex set such that the color of vertex $i$ is the same as the color of $\graphautomorphism(i)$, and vertices $(i,j)$ form an edge if and only if $\left(\graphautomorphism(i),\graphautomorphism(j)\right)$ also form an edge.
\end{definition}

By finding the automorphisms of the graph $\incidencegraph$
we can constrain our search for automorphisms of 
$\{{\cal V}_3,{\cal V}_1\}$, and thus of $\LL$.

\begin{lemma}
\label{lem:graphauto}
The bipartite graph $\incidencegraph$ has $11520$ automorphisms.
Under this automorphism group, the graph has three orbits. One orbit corresponds to the set of $60$ three-dimensional singular subspaces. 
Another orbit corresponds to the subset of $30$ one-dimensional subspaces
in ${\cal V}_1$
of type I and II.
A third orbit corresponds to the subset of $16$ one-dimensional subspaces 
of type III.
\end{lemma}

\begin{proof}
We have computed this using Nauty~\cite{nty}
within Magma.
(See supplemental script.)
\end{proof}

\subsubsection{Graph automorphisms to arrangement automorphisms}
A priori, it might be the case that some of these
graph automorphisms do not arise from a linear transform of $\CC^6$
act as an automorphism on the \emph{subspace arrangement} $\{{\cal V}_3, {\cal V}_1\} \subset \CC^6$. We rule this out.

\begin{lemma}
\label{lem:l24Lin}
Each of the graph automorphisms of $\incidencegraph$ gives rise to  a unique   linear automorphism of the arrangement
$\{{\cal V}_3,{\cal V}_1\}$
on $\LL$, up to a global scale.
Each equivalence class of such linear maps contains a rational-valued matrix. 
\end{lemma}
\begin{proof}
Each graph automorphism $\graphautomorphism$ 
gives rise to a permutation of the spaces in ${\cal V}_3$. A $6\times 6$ matrix $\A$ describing a linear transform that maps the three-dimensional subspaces in the same manner must satisfy $540=60\cdot 9$ linear homogeneous constraints, nine for each pair $\left(i,\graphautomorphism(i)\right), i\in {\cal V}_3$. 

Magma gives us a generating set of size $6$ for the group
of graph automorphisms.
For each of the $6$ generators of the graph automorphism group, we write out the system of linear constraints. When doing so, we discover that this system always has a solution that is unique, up to a global scale. The $540\times 36$ constraint matrix can always be written as a rational-valued matrix, since the subspace arrangement $\{{\cal V}_3, {\cal V}_1\}$ can be defined using rational-valued coefficients.
(See supplemental script.)
\end{proof}

\subsubsection{Arrangement automorphisms are $L_{2,4}$ automorphisms}

It might also be possible that there are 
linear transforms which preserve 
the subspace arrangement $\{{\cal V}_3, {\cal V}_1\}$, but do 
not preserve the entire $\LL$ variety.  We rule this out as well.

\begin{lemma}
\label{lem:l24Auto}
Each of the graph automorphisms of $\incidencegraph$ gives rise to  a unique   linear automorphism
on $\LL$, up to a global scale.
Each equivalence class of such linear maps contains a rational-valued matrix. 
\end{lemma}
\begin{proof}
From Lemma~\ref{lem:l24Lin}, each of the graph automorphisms
gives rise to a, unique up to scale, rational-valued
linear automorphism of our arrangement.
When we pull back the single defining equation of $\LL$ through each such invertible linear map,
we verify that we recover said equation. Thus this map
is a linear automorphism of $\LL$.
\end{proof}

\subsubsection{Reflection group}
Next, we make a definition that will be helpful in establishing the connection between 
$\PP_{+}\Aut_{\RR}(L_{2,4})$
and the Weyl group $D_6$. For definitions, see~\cite{humphreys}.

\begin{definition}
We define the \defn{reflection group} $\reflectiongroup$ as the real
matrix group generated by the set of reflections in $\RR^6$ across the 
$30$ hyperplanes that are orthogonal to the $30$ one-dimensional real intersection subspaces of type I and II.
\end{definition}

The following lemma was based on conversions with
Dylan Thurston.

\begin{lemma}
\label{lem:reflgrp}
The reflection group $\reflectiongroup$ is of order $23040$, and is isomorphic to the Weyl group $D_6$. The reflection group leaves the variety $\LL$ invariant.
\end{lemma}

\begin{proof}
From the $30$ vectors that generate $\reflectiongroup$, we generate a larger set of $60$ vectors $\phi$ that has the same reflection group as follows: For each vector $\f$ in the original 30-set, we create two vectors $\pm 2\f/\Vert\f\Vert$ in the 60-set. 
Next, we verify that the set $\phi$ is a (reduced, crystallographic) root system by: i) applying each generator of the  group $\reflectiongroup$ to the set $\phi$ and verifying that it leaves the set invariant; and ii) verifying that the set satisfies the integrality condition $\forall \f,\g \in \phi, 2(\f\cdot \g)/\Vert\f\Vert \in \ZZ$. 

A reflection group of a root system is a Weyl group. To prove the first part of the lemma, we need only classify the root system (and thus the Weyl group) according to the finite catalog of rank 6 possibilities. We use the procedure described in \cite[page 48]{humphreys}, which we summarize here. 

We begin by choosing any vector $\h \in \QQ^6$ that is not proportional or perpendicular to a vector in $\phi$, and then we identify the subset of \emph{positive roots} $\phi^+:=\{\f : \f\in \phi, (\h\cdot \f)>0\}$. Since $\phi$ is a root system, it will be the case that $|\phi^+| = |\phi|/2 = 30$. Among the positive roots, we identify the subset of \emph{simple roots} as the vectors $\f \in \phi^+$ that cannot be decomposed as $\g_1+\g_2$ for some $\g_i\in\phi^+$. By construction, simple roots form a basis for the embedding vector space, so in the present case there will be $6$ of them. Finally, we can classify the group by examining the pattern of pairwise angles between simple roots.

Applying this calculation to our root system, we find that the
pairwise angles between the simple roots are $0$ or $2\pi/3$. We draw
a Dynkin diagram that has one vertex for each simple root and
an edge $(i,j)$ whenever the angle between roots $i$ and $j$ is
$2\pi/3$.  Doing so, we find that this diagram is of type $D_6$. This means  that the reflection group is isomorphic to the Weyl group $D_6$, which is
of order $23040$. This proves the first part of the lemma.

To prove the second part of the lemma, we use the fact that the reflection group $\reflectiongroup$ is generated by the 6 reflections from the simple roots. We pull back the single defining equation of $\LL$ through each of these 6 linear maps, and we verify that we recover said equation.

Note that the group could also be identified from its
computed order.
(See supplemental script.)
\end{proof}

\subsubsection{Proof}
The proof of our theorem is  now nearly complete.

\begin{proof}[Proof of Theorem~\ref{thm:auto2}]
From Theorem~\ref{thm:Tmap}, a linear automorphism of
$\LL$ must be a linear automorphism of its singular set
${\cal V}_3$, and thus must preserve the incidence
structure of 
$\{{\cal V}_3,{\cal V}_1\}$. Any linear automorphism
of this incidence structure must give rise to a
graph automorphism of $\incidencegraph$.
By Lemma~\ref{lem:graphauto}, there 
are $11520$ graph automorphisms of $\incidencegraph$, and from 
Lemma~\ref{lem:l24Auto}, each gives rise to a 
rational valued linear automorphism of $\LL$, unique up to scale. 
Summarizing, we have shown that $\PP \Aut(L_{2,4}) = \PP \Aut(\sing(L_{2,4}))$,
and that both of these groups are isomorphic to the automorphism group of the 
graph $\Delta$.  Lemma~\ref{lem:l24Auto} also implies that each equivalence 
class in $\PP \Aut(L_{2,4})$ contains a rational representative, so 
this group can be generated by rational matrices.

Because of the rational generators mentioned above, the group 
$\PP \Aut_{\RR}(L_{2,4})$ is isomorphic to the others.  It then follows that
the order of $\PP_{+} \Aut_{\RR}(L_{2,4})$ is $23040 = 2\cdot 11520$.

Next, we deal with the classification of $\PP_{+} \Aut_{\RR}(L_{2,4})$.
By Lemma~\ref{lem:reflgrp} (specifically the second statement), the 
elements of $W$ generate some subgroup $G$ of  $\PP_{+} \Aut_{\RR}(L_{2,4})$.
In fact,  no two elements of $\reflectiongroup$ are related by a positive scale,
so $W$ is isomorphic to this $G$.  The first part of Lemma~\ref{lem:reflgrp} says that
$W$ has the same order as  $\PP_{+} \Aut_{\RR}(L_{2,4})$, so $W$ and $\PP_{+} \Aut_{\RR}(L_{2,4})$ are isomorphic.

For the third part of the theorem, we need only test 
$23040$ matrices and retain those that have only non-negative entries. 
This has been done in the Magma CAS, and indeed, it yields only the $24$ 
edge permutations induced by vertex relabeling.  
(See supplemental script.)
This is, in particular, 
a subgroup of $\PP_{+} \Aut_{\RR}(L_{2,4})$.

\end{proof}

\subsection{Automorphisms of $L_{1,3}$}

\begin{theorem}\label{thm:l13-aut}
Any linear automorphism 
$\A$ of $L_{1,3}$  is a scalar multiple
of a signed permutation that is induced by a 
vertex relabeling.
\end{theorem}
\begin{proof}
$L_{1,3}$ comprises $4$ hyperplanes.
Each permutation on these $4$ planes gives us
at most a single linear automorphism of $L_{1,3}$ up to scale.
Thus $\PP\Aut(L_{1,3})$ is isomorphic to a subgroup of $S_4$
and, in particular, has order at most $24$.

Meanwhile $\PP\Aut(L_{1,3})$ contains a subgroup of order 
$24$ generated by vertex relabeling and sign flips.  By 
the above, this must be the whole group.
\end{proof}
\begin{remark}
If we want to see $S_4$ acting by sign flips and 
coordinate permutations, we can observe
that these maps are symmetries of the cube that 
permute the opposite corner diagonals.  
\end{remark}

\appendix

\section{Algebraic Geometry Preliminaries}\label{sec:geometry}

We summarize the needed definitions and facts about
complex algebraic varieties. For more see~\cite{harris}.

In this section $N$ and $D$ will represent arbitrary numbers.

\begin{definition}
A (complex embedded affine) \defn{variety} 
(or \defn{algebraic set}), $V$,
is a (not necessarily strict)
subset of $\CC^N$, for some $N$,
that is defined by the simultaneous
vanishing of a finite set of polynomial equations 
with coefficients in $\CC$
in the 
variables $x_1, x_2, \ldots, x_N$ which are associated with the 
coordinate axes of $\CC^N$.

A variety can be stratified as a union of a finite number of
complex analytic submanifolds of $\CC^N$.

A finite union of varieties is a variety.
An arbitrary intersection of varieties is a variety.

The set of polynomials that vanish on $V$ form 
a radical ideal $I(V)$, which is generated by a finite set
of polynomials.

A variety $V$ is \defn{reducible} if it is the proper union of two
varieties $V_1$ and $V_2$. 
(Proper means that $V_1$ is not contained in $V_2$ and vice versa.)
Otherwise it is called
\defn{irreducible}.
A variety has a unique decomposition as a finite proper
union of its
maximal irreducible subvarieties called \defn{components}.
(Maximal means that a component cannot be 
 contained in a larger  irreducible 
subvariety of $V$.)

A variety $V$ has a well defined (maximal) \defn{dimension} $\Dim(V)$, 
which will agree with the largest $D$ for which there
is an open subset of~$V$, in the standard  topology, 
that is a $D$-dimensional complex submanifold of $\CC^N$.

The \defn{local dimension} $\Dim_\genericpoint(V)$ at a point $\genericpoint$ is the
dimension of the highest-dimensional irreducible component of $V$ that contains $\genericpoint$.
If all components of $V$ have the same dimension, we say it has
\defn{pure dimension}.

Any strict subvariety $W$ of an 
irreducible variety $V$ must be of strictly lower dimension.


\end{definition}

\begin{definition}
A \defn{constructible set} $S$ is a set that can be defined using a finite
number of varieties and a finite number of Boolean set operations.

The \defn{Zariski closure} of $S$ is the smallest variety $V$ containing it.
The set $S$ has the same dimension as its Zariski closure $V$.

The image of a variety $V$ of dimension $D$
under a polynomial map is a constructible set $S$ of dimension
at most $D$.
If $V$ is irreducible, then so 
too is the Zariski closure of $S$.
(We say that $S$
is \defn{irreducible}.)
\end{definition}

\begin{theorem}
\label{thm:Zdense}
Any variety $V$ is a closed subset of $\CC^N$ in the standard 
topology.
If a subset $S$ of $\CC^N$
is standard-topology dense in a variety $V$, then
$V$ is the Zariski closure of $S$.
\end{theorem}

We will need the following easy lemmas.

\begin{lemma}
\label{lem:bij}
Let $\A$ be a bijective linear map 
on $\CC^N$. The image under $\A$
of a variety $V$ is a variety
of the same dimension.
If $V$ is irreducible, then so too is this image.
\end{lemma}
\begin{proof}
The image $S := \A(V)$ 
must be a constructible set.

Since $\A$ is bijective, then
there is also map $\A^{-1}$ acting on $\CC^N$,
and $S$ must be the inverse image
of $V$ under this map.
Thus, by pulling back the defining equations of $V$ through
$\A^{-1}$, we see that $S$ must also be a variety.

The dimension  follows 
from the fact that maps cannot raise dimension, and our map
is invertible.
\end{proof}

\begin{theorem}
\label{thm:comps}
If $\A$ is a bijective linear map on $\CC^N$
that acts as 
bijection between two reducible varieties
$V$ and $W$, then it must bijectively map components
of $V$ to components of $W$.
\end{theorem}
\begin{proof}
From Lemma~\ref{lem:bij},
$\A$ must
map irreducible varieties to irreducible varieties.
As a bijection, it also must preserve subset relations
(which define maximality).
\end{proof}

\begin{lemma}
\label{lem:comp}
Let $V = V_1 \cup V_2$ be a union of varieties.
Then any irreducible subvariety $W$ of $V$ must be 
fully contained in at least one of the
$V_i$.
\end{lemma}
\begin{proof}
If $W$ was not fully contained in either $V_i$, then it could
be written as the proper union of varieties $W = \bigcup_i (W\cap V_i)$
contradicting its irreducibility.
\end{proof}



There are two approaches for defining smooth and singular points.
One comes from our  
algebraic setting, while the other comes from the more general setting
of complex analytic varieties (which we will explicitly 
refer to as ``analytic''). 
It will turn out that (algebraic) smoothness implies analytic smoothness,
and  that analytic smoothness
implies (algebraic) smoothness.

\begin{definition}
The \defn{Zariski tangent space} at a point $\genericpoint$ of a
variety $V$
is the kernel of the Jacobian matrix of a set of
generating polynomials for $I(V)$ evaluated at $\genericpoint$.



A point $\genericpoint$ is called
(algebraically) \defn{smooth} in $V$ if the dimension of the Zariski
tangent space equals the local dimension $\Dim_\genericpoint(V)$. Otherwise $\genericpoint$ is called
(algebraically) \defn{singular} in $V$.
The \defn{locus} of singular points of $V$ is denoted $\sing(V)$.
The singular locus is itself a strict subvariety of $V$. 


\end{definition}

\begin{theorem}
\label{thm:Tmap}
If $\A$ is a bijective linear map on $\CC^N$ that acts as a 
bijection between two irreducible varieties
$V$ and $W$, then it must map singular points to singular points.
\end{theorem}
This is a special case of the more general setting of 
``regular maps'' and 
``isomorphisms of varieties''~\cite[Page 175]{harris}.

\begin{theorem}
\label{thm:conIrr}
If a point $\genericpoint$ is contained in two distinct components of $V$,
then $\genericpoint$ cannot be 
a  smooth point in $V$. 
\end{theorem}
See~\cite[II. 2. Theorem 6]{shaf}.

\begin{definition}
If a point $\genericpoint$ in a variety
$V$ 
has a neighborhood
in $V$ that is a complex submanifold of $\CC^N$ with  some dimension 
$D$,
then we call the point
\defn{analytically smooth of dimension $D$} in $V$, or just
\defn{analytically smooth} in $V$. Otherwise we call the point
\defn{analytically singular} in $V$. 
\end{definition}

The following theorem tells us that there is no difference between these to notions of smoothness.
\begin{theorem}
\label{thm:smpt}
An (algebraically) smooth point $\genericpoint$ in a variety $V$ must be an
analytically  smooth point of dimension $\Dim_\genericpoint(V)$ in $V$.

A point $\genericpoint$ 
that is analytically smooth of dimension $D$
in $V$
must be
an (algebraically) smooth point $\genericpoint$ in $V$ with
$\Dim_\genericpoint(V)=D$.
\end{theorem}
For discussions on this theorem see~\cite[Exercise 14.1]{harris},
~\cite[Page 13]{milnor}. See~\cite[Page 14]{sing} for the setting
where one does not assume irreducibility, or even pure dimension.

Note that the second direction does not have a corresponding statement in the setting of real algebraic varieties.

\section{Fano Varieties of $\LL$}
\label{sec:fano}

This section contains a bonus result about 
the linear subsets in 
$L_{2,4}$. Though it is not needed for the rest of
the paper, it can be of use for unlabeled rigidity
problems~\cite{loops0}.

\begin{definition}
\label{def:fano}
Given an affine algebraic cone $V \subset \CC^{N}$ (an affine variety defined by a 
homogeneous ideal), its \defn{Fano-$k$} variety 
$\Fano_k(V)$
is the subset
of the Grassmanian $\Gr(k+1,N)$ corresponding to $k+1$-dimensional linear
subspaces that are contained in $V$. 
\end{definition}

\begin{theorem}
\label{thm:3flats}
The only $3$-dimensional linear subspaces that are contained in $\LL$ are the $60$ $3$-dimensional linear spaces comprising its singular locus.
Moreover, there are no linear subspaces of dimension
$\ge 4$ contained in $\LL$.
\end{theorem}
\begin{proof}
This proposition is proven by calculating  the Fano-$2$ variety of 
$\LL$ in the Magma CAS~\cite{magma}, and comparing it to the the Fano-$2$ variety of 
the singular locus of $\LL$. 

We use the approach described in~\cite[Page 70]{harris} to compute 
the $\Fano_2(\LL)$ variety. We summarize this approach here.
We shall order the coordinates of $\CC^6$ in the order
$( l_{12}, l_{13}, l_{23}, l_{14}, l_{24}, l_{34})$.

Let us specify a point in $\CC^6$ as 
\ba
\begin{pmatrix}
1&0&0 \\
0&1&0 \\
0&0&1 \\
\lambda_1&\lambda_2&\lambda_3\\ 
\lambda_4&\lambda_5&\lambda_6\\ 
\lambda_7&\lambda_8&\lambda_9 
\end{pmatrix}
\begin{pmatrix}
t_1\\
t_2\\
t_3
\end{pmatrix}
\ea
where the $\lambda_i$ are variables that specify a three-dimensional
linear subspace of $\CC^6$, and the $t_j$ are variables that
specify a point
on that subspace. Note that this can only represent
an affine open subset of the Grassmanian; it cannot represent three-dimensional
linear subspaces that are parallel to the first three coordinate axes.

We can compute the polynomial in $[\lambda_i, t_j]$
vanishing when the associated 
points in $\CC^6$ are also in $\LL$. We can then 
look at all of the coefficients
(polynomials in $\lambda_i$) of the monomials in $t_j$.
These coefficient polynomials vanish identically iff 
the linear subspace specified by the $\lambda_i$ is in $\LL$.
Thus these coefficients generate an affine open subset of
$\Fano_2(\LL)$.

To study the whole Fano variety, we must also 
look at the other affine subsets of the Grassmanian. Due to 
the vertex symmetry of $\LL$,
we only need to consider the additional two matrices:

\ba
\begin{pmatrix}
1&0&0 \\
0&1&0 \\
\lambda_1&\lambda_2&\lambda_3\\ 
0&0&1 \\
\lambda_4&\lambda_5&\lambda_6\\ 
\lambda_7&\lambda_8&\lambda_9 
\end{pmatrix}
\;\;{\rm and}\;\; 
\begin{pmatrix}
1&0&0 \\
\lambda_1&\lambda_2&\lambda_3\\ 
0&1&0 \\
\lambda_4&\lambda_5&\lambda_6\\ 
\lambda_7&\lambda_8&\lambda_9 \\
0&0&1 
\end{pmatrix}
\ea

These three matrices represent the triplet of coordinate axes corresponding to,
respectively, a triangle, a chicken-foot, and a simple open path. Thus, these $3$ open subsets of 
$\Fano_2(\LL)$, together with vertex relabeling, cover the full Fano variety.

We compute these $3$ open subsets of $\Fano_2(\LL)$ in Magma, and verify that, in each of these open subsets,
$\Fano_2(\LL)$ is $0$-dimensional and 
$|\Fano_2(\LL)| =|\Fano_2(\sing(\LL))|$. 
As  
$\Fano_2(\LL) \supset \Fano_2(\sing(\LL))$, 
we can conclude that  
$\Fano_2(\LL) =\Fano_2(\sing(\LL))$ (see supplemental script).

As Fano-$2$ variety is discrete, the higher Fano
varieties of $\LL$ must also be empty.
\end{proof}

\begin{remark}
\label{rem:fano3d}
We have been unable to fully
compute any of the
Fano varieties of $L_{3,5}$ in any computer algebra system, but partial
results do not look promising.
We have been able to verify that  
$\Fano_6(L_{3,5})$ 
is not empty (see supplemental script).
This together with our (partial) understanding
of $\sing(L_{3,5})$ suggests that 
$L_{3,5}$ indeed
contains $6$-dimensional linear spaces that are not 
contained in its singular locus. 
\end{remark}



\bibliographystyle{abbrvnat}
\bibliography{loops}

\end{document}